\documentclass[11pt,a4paper]{amsart}
\input epsf

\usepackage{graphicx, amsmath, amssymb, amscd, amsthm, euscript, amsfonts, xcolor, latexsym, amsbsy}

\usepackage{hyperref}
\usepackage{float}
\usepackage{longtable}
\usepackage{subfigure}
\usepackage{graphics, geometry, caption}
\usepackage{tikz}
\usetikzlibrary{decorations.pathmorphing}

\numberwithin{equation}{section}

\newcommand{\sech}{{\rm sech\,}}

\newtheorem{thm}{Theorem}[section]

\newtheorem{prop}[thm]{Proposition}
\newtheorem{lem}[thm]{Lemma}
\theoremstyle{definition}
\newtheorem{defn}[thm]{Definition}

\newtheorem{rem}[thm]{Remark}

\title[complex hyperbolic ultra-parallel triangle groups]{Discreteness of  complex hyperbolic ultra-parallel triangle groups}

\author[W. Liao]{Wei Liao}
\address[W. Liao]{School of Mathematics, Hunan University, Changsha, P.R. China}
\email{liaowei.math@qq.com}

\author[B. Xie]{Baohua Xie}
\address[B. Xie]{School of Mathematics, Hunan University, Changsha, P.R. China}
\address{Greater Bay Area Institute for Innovation, Hunan University, Guangzhou, China}
\email{xiexbh@hnu.edu.cn}


\begin{document}

\begin{abstract}
We prove that a family of complex hyperbolic $[m_1, m_2, m_3]$-triangle group representations with $m_3 > 0$ is discrete and faithful if and only if the isometry $R_1(R_2R_1)^nR_3$ is non-elliptic for some positive integer $n$. Additionally, we investigate the case where $m_3 = 0$ and, in this context, give improved bounds on the discreteness locus compared to those given by Monaghan, Parker, and Pratoussevitch in 2019.
\end{abstract}
\maketitle

\section{Introduction}
A subgroup of \( \mathrm{PU}(2,1) \) is called a \emph{complex hyperbolic ultra-parallel triangle group} if it is generated by three complex reflections \( R_1, R_2, R_3 \), whose mirrors \( C_1, C_2, C_3 \) are pairwise disjoint in the complex hyperbolic space \( \mathbf{H}_{\mathbb{C}}^2 \). After arranging the mirrors appropriately, we define \( m_j \) as the distance between the mirrors \( C_{j-1} \) and \( C_{j+1} \) for \( j = 1,2,3 \pmod{3}  \), subject to the condition \( m_1 \geq m_2 \geq m_3 \geq 0 \). This group is denoted as the complex hyperbolic ultra-parallel \( [m_1, m_2, m_3] \)-triangle group.

A \emph{complex hyperbolic $[m_1, m_2, m_3]$-triangle group representation} is a representation $\rho$ of the group  
\[
\Gamma = \langle \iota_1, \iota_2, \iota_3 \mid \iota_j^2 = 1,\ j = 1, 2, 3 \rangle
\]
into $\mathrm{PU}(2,1)$ such that $\rho(\iota_j)$ is a complex reflection for each $j = 1, 2, 3$, and the group generated by $R_j = \rho(\iota_j)$ forms a complex hyperbolic ultra-parallel $[m_1, m_2, m_3]$-triangle group.  
A representation $\rho$ is called \emph{discrete} if $\rho(\Gamma)$ is discrete.

The study of discrete and faithful complex hyperbolic $[m_1, m_2, m_3]$-triangle group representations was initiated by Goldman and Parker \cite{gp}. They investigated representations of type \( [0,0,0] \) and partially classified them. Furthermore, they conjectured that such a representation is discrete and faithful if and only if the product of the three generators is non-elliptic. This conjecture, known as the Goldman-Parker conjecture, was later confirmed by Schwartz \cite{sch1}. A new proof was subsequently provided in the combined works of \cite{sch2} and \cite{sch3}.

In the decades since, various types of complex hyperbolic $[m_1, m_2, m_3]$-triangle group representations have been studied. Examples include representations of type \( [2m_3, m_3, m_3] \) \cite{wg}, type \( [m_1, m_1, m_1] \) \cite{vas}, type \( [m_1, m_2, 0] \) \cite{mon, mpp}, and the higher-order case of type \( [m, m, 0; n_1, n_2, 2] \) \cite{pov, pp1, pp2}.

Additionally, the study of another class of triangle group representations has attracted the attention of many researchers: the \emph{complex hyperbolic \( (p_1, p_2, p_3) \)-triangle group representations}, where \( 3 \le p_1 \le p_2 \le p_3 \le \infty \). A \emph{complex hyperbolic \( (p_1, p_2, p_3) \)-triangle group} is a subgroup of \( \mathrm{PU}(2, 1) \) generated by three complex reflections \( R_1 \), \( R_2 \), and \( R_3 \) with angles \( \theta_j \) between the mirrors \( C_{j-1} \) and \( C_{j+1} \), where $\theta_j = \pi/p_j$ for $j = 1, 2, 3 \pmod{3}$. A complex hyperbolic \( (p_1, p_2, p_3) \)-triangle group representation \( \rho \) is a representation of the group
\[
\langle \iota_1, \iota_2, \iota_3 : \iota_j^2 = (\iota_2\iota_3)^{p_1} = (\iota_1\iota_3)^{p_2} = (\iota_1\iota_2)^{p_3} = 1,\,\, j = 1, 2, 3 \rangle
\]
into \( \mathrm{PU}(2, 1) \), such that \( \rho(\iota_j) \) corresponds to the complex reflection \( R_j \) for \( j = 1, 2, 3 \), and the group generated by \( R_j \) forms a complex hyperbolic \( (p_1, p_2, p_3) \)-triangle group.

In \cite{sch4}, Schwartz classified complex hyperbolic $(p_1, p_2, p_3)$-triangle group representations into two types, denoted type $A$ and type $B$. Specifically, a representation is called type $A$ if the element $W_A = R_1 R_3 R_2 R_3$ becomes elliptic before the element $W_B = R_1 R_2 R_3$, and type $B$ if $W_B$ becomes elliptic before $W_A$. Schwartz conjectured that a type $A$ representation is discrete and faithful if and only if $W_A$ is non-elliptic, and a type $B$ representation is discrete and faithful if and only if $W_B$ is non-elliptic. In \cite{gro}, Grossi proved that a complex hyperbolic triangle group representation is of type $A$ if $p_1 < 10$ and of type $B$ if $p_1 > 13$. 

Although a complete proof of Schwartz's conjecture remains elusive, several partial results have been established. Specifically, representations of ideal triangle groups were studied in \cite{sch1, sch2, sch3}, while results for $(3,3,n)$, $(3,3,\infty)$, $(3,n,\infty)$, and $(4,4,\infty)$ triangle groups were obtained in \cite{pwx, pw, xwx, jwx}, respectively. Other discreteness results for unfaithful complex hyperbolic triangle group representations can be found in \cite{der, dpp1, dpp2, kpt1, kpt2, par2, pp, sch5, tho}.

According to the work of Monaghan, Parker and Pratoussevitch \cite{mpp}, the case of complex hyperbolic $[m_1, m_2, m_3]$-triangle group representations differs from that of complex hyperbolic $(p_1, p_2, p_3)$-triangle group representations. In general, whether a complex hyperbolic $[m_1, m_2, m_3]$-triangle group representation is discrete and faithful is  not necessarily determined by the types of isometries \( W_A \) and \( W_B \). This is because Proposition 1 in \cite{mpp} proves that for a large family of complex hyperbolic \( [m_1, m_2, 0] \)-triangle group representations, these representations are discrete and faithful if and only if the element \( w^{(n)} = R_1 (R_2 R_1)^n R_3 \) is non-elliptic for some positive integer \( n \).

In this paper, we primarily focus on the case where \( m_3 > 0 \) and provide a family of complex hyperbolic $[m_1, m_2, m_3]$-triangle group representations. These representations are discrete and faithful if and only if the element \( w^{(n)} \) is non-elliptic for some positive integer \( n \). Additionally, we consider the case where \( m_3 = 0 \) and provide an improvement of the result in \cite{mpp}. In contrast to their approach, the proof presented in this paper is more naturally geometric.

The result is analogous to bounding the deformation space of Fuchsian Schottky groups of genus $2$ with two allowed parabolic elements (i.e.\ Riemann surfaces of genus $2$ with a reflection plane cutting them into three circles, two of which are allowed to be pinched to points). In this analogy, the three parameters $[m_1, m_2, m_3]$ correspond to the real translation lengths of the three boundary hyperbolic elements, and $m_3>0$ corresponds to the condition that one of the circles is not pinched to a point. There is a great deal of interest in the literature on discreteness bounds for $2$-generated Fuchsian groups of this kind going back a long way; see, for instance, Purzitsky~\cite{pur}, Rosenberger~\cite{ros}, and Gilman--Keen~\cite{gk}.

We now present our first main result.

\begin{thm}\label{main thm 2}
    Let \( m_3 > 0 \). There exists a family of semi-analytic sets \( K_n \subset \mathbb{R}^3 \) indexed by \( n \in \{1, \ldots, k_0\} \) with the following properties:
    \begin{enumerate}
    \item[(a)] Every point \((m_1, m_2, \alpha) \in K_n\) corresponds, via an explicit parameterization given in Proposition~\ref{angular}, to three elements \( R_1, R_2, R_3 \in \mathrm{PU}(2, 1) \) such that \(\langle R_1, R_2, R_3 \rangle\) is a representation of a complex hyperbolic \([m_1, m_2, m_3]\)-triangle group with angular invariant \(\alpha\).
    \item[(b)] Each \(K_n\) is defined by \(4+k_0\) inequalities in the quantities \( r_j = \cosh(m_j/2) \) and \( s_j = \sinh (m_j/2) \), as given in Definition~\ref{defn of Kn}.
    \item[(c)] A representation corresponding to a point of \(K_n\) is discrete and faithful if and only if \(R_1(R_2R_1)^nR_3\) is non-elliptic.
\end{enumerate}
\end{thm}

As examples, for $r_3 = 1.01$ (resp. $r_3 = 1.09$) and $k_0 = 3$, we present in Figure~\ref{ultra-picture1.png} (resp. Figure~\ref{ultra-picture2.png}) the plots of the sets $K_1$, $K_2$, and $K_3$ in the $(X, Y)$-coordinates
\begin{equation}\label{(r1, r2)-(X, Y)}
    (X, Y) = \left(\frac{r_1^2 - 1}{r_2^2 - 1} - 1, \frac{1}{r_2^2 - 1}\right),
\end{equation}
as introduced in \cite{mpp}.
Any representations corresponding to the dark region are faithful and discrete for any choice of the angular invariant $\alpha$.
Additionally, computational experiments suggest that even as \( r_3 \) and \( k_0 \) increase, non-empty sets \( K_n \) continue to exist.

We then provide our second result, which gives an improvement over Proposition~1 in \cite{mpp}.

\begin{thm}\label{main thm 3}
Assume that \( m_2 > m_3 = 0 \). There exists a family of semi-analytic sets \( \mathcal{K}_n \subset \mathbb{R}^3 \) for \( n \in \mathbb{Z}^+ \) with the following properties:
\begin{enumerate}
    \item[(a)] Every point \( (m_1, m_2, \alpha) \in \mathcal{K}_n \) corresponds, via the explicit parameterization given in Section~\ref{sec m3=0}, to three elements \( R_1, R_2, R_3 \in \mathrm{PU}(2,1) \) such that \( \langle R_1, R_2, R_3 \rangle \) is a representation of a complex hyperbolic \( [m_1, m_2, 0] \)-triangle group with angular invariant \( \alpha \).
    \item[(b)] Each set \( \mathcal{K}_n \) is defined by three inequalities in the quantities \( r_j = \cosh\!\left(m_j/2\right) \) and \( s_j = \sinh\!\left(m_j/2\right) \), as given in Definition~\ref{defn of KN}.
    \item[(c)] A representation corresponding to a point in \( \mathcal{K}_n \) is discrete and faithful if and only if the element \( R_1 (R_2 R_1)^n R_3 \) is non-elliptic.
\end{enumerate}
\end{thm}

Define
\begin{equation}\label{Pn}
 \mathcal{P}_n
=
\left\{
(X, Y, \alpha) :
\frac{2}{n} \le X \le \frac{2}{n-1}, \ 
nX - n(n+1)Y - 1 \ge 0, \ 
\alpha \in [0, \pi]
\right\},
\end{equation}
where the inequality \( X \le \frac{2}{n-1} \) is omitted in the case \( n = 1 \), and the correspondence between \( (r_1, r_2) \) and \( (X, Y) \) is given in equation~\eqref{(r1, r2)-(X, Y)}. Replacing the set \( \mathcal{K}_n \) in Theorem~\ref{main thm 3} by \( \mathcal{P}_n \), this statement coincides exactly with Proposition~1 of \cite{mpp}. A comparison with Proposition~1 of \cite{mpp} is illustrated in Figure~\ref{ultra for m3=0.png}.

In Section~\ref{sec 7}, we further show that for any given positive integer \( n \ge 2 \), the set \( \mathcal{P}_n \) is the limit, as \( m_3 \) tends to \( 0 \), of an explicit subset of \( K_n \) given in Theorem~\ref{main thm 2}. For \( n = 1 \), we show that any compact subset of \( \mathcal{P}_1 \) is the limit of a subset of \( K_1 \) as \( m_3 \) tends to \( 0 \).

\begin{figure}[htbp]
    \centering
    \includegraphics[scale=0.7]{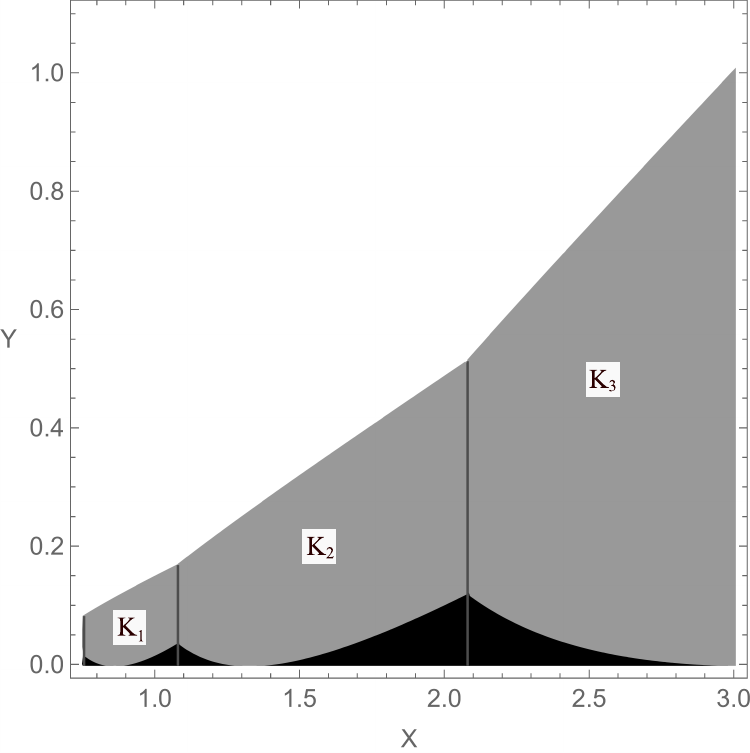}
    \caption{The sets \( K_n \) for \( n = 1, 2, 3 \) with $r_3 = 1.01$ and $k_0 = 3$.}
    \label{ultra-picture1.png}
\end{figure}

\begin{figure}[htbp]
    \centering
    \includegraphics[scale=0.7]{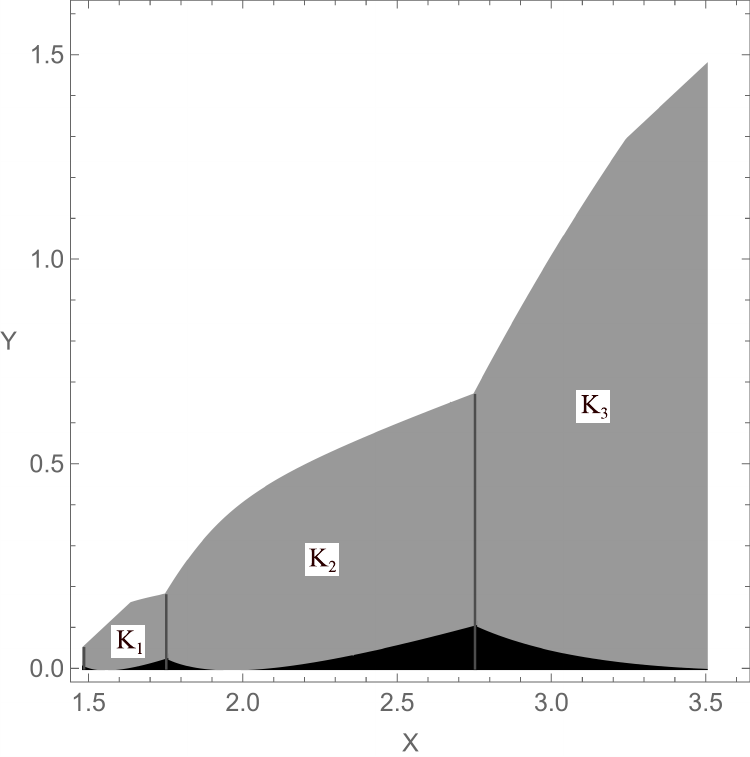}
    \caption{The sets \( K_n \) for \( n = 1, 2, 3 \) with $r_3 = 1.09$ and $k_0 = 3$.}
    \label{ultra-picture2.png}
\end{figure}

\begin{figure}[htbp]
    \centering
    \includegraphics[scale=0.69]{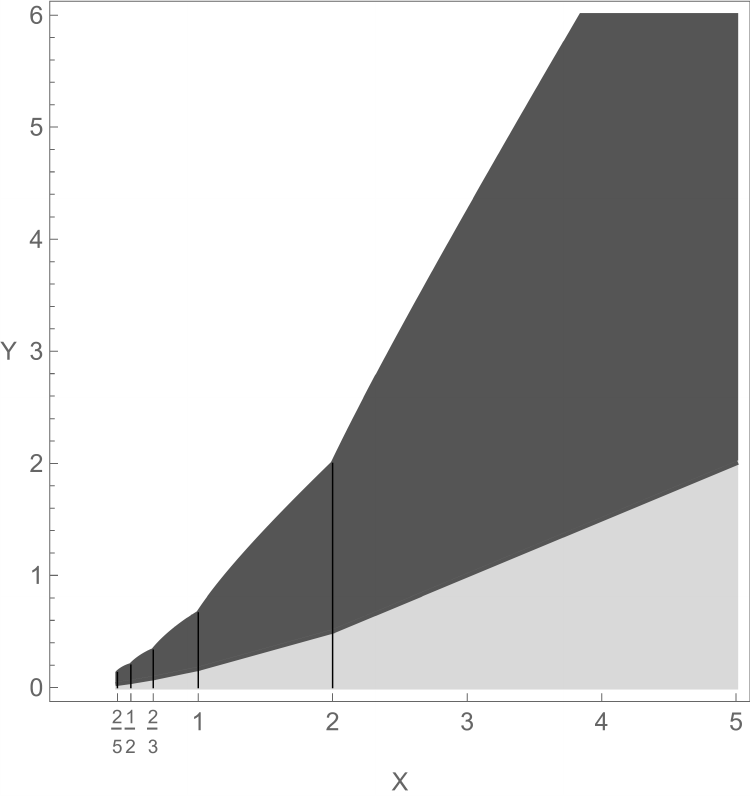}
    \caption{Within each strip, defined by $\frac{2}{n} \le X \le \frac{2}{n-1}$ for $n \in \mathbb{Z}^+$ (with the upper bound omitted for $n = 1$), the combined colored regions (both dark and light gray) represent the set $\mathcal{K}_n$ in Theorem~\ref{main thm 3}. Specifically, the light gray region alone corresponds to the result in Proposition~1 of \cite{mpp}.}
    \label{ultra for m3=0.png}
\end{figure}

In recent decades, Klein's combination theorem has found extensive applications in complex hyperbolic geometry; see, for example, \cite{gp, jx, kp, lx, mon, mpp, wg, sch1, sch2, sch3, xj}. A recent application in real hyperbolic geometry is discussed in \cite{egms}.
We prove Theorem~\ref{main thm 2} and Theorem~\ref{main thm 3} using a variant of Klein's combination theorem. In the proofs of Theorems~\ref{main thm 2} and~\ref{main thm 3}, we use half-spaces bounded by appropriate bisectors or Cygan spheres as fundamental domains for the groups generated by a complex reflection, respectively. For this bisector, we require that the mirror of the complex reflection be a $\mathbb{C}$-slice of the bisector. This construction is motivated by \cite{gil} in real hyperbolic geometry and \cite{rxj} in complex hyperbolic geometry. Other constructions of fundamental domains for groups generated by a single complex reflection also exist. For example, in \cite{sch1} and \cite{wg}, the fundamental domains considered are those bounded by Clifford tori and $\mathbb{R}$-spheres.

This paper is organized as follows: In Section~\ref{sec basic}, we introduce fundamental concepts of complex hyperbolic geometry. In Section~\ref{sec para}, we present a parameterization of complex hyperbolic $[m_1, m_2, m_3]$-triangle group representations. In Section~\ref{sec discreteness result}, we establish a new discreteness result using a variant of Klein's combination theorem. In Sections~\ref{sec m3>0} and~\ref{sec m3=0}, we prove Theorem~\ref{main thm 2} and Theorem~\ref{main thm 3}, respectively. Finally, in Section~\ref{sec 7}, we consider the limiting behaviour of the subset of $K_n$ as $m_3$ tends to $0$.
\medskip

\section{Background}\label{sec basic}
In this section, we provide the basics of complex hyperbolic geometry, most of which can be found in \cite{gol}. For an introduction to complex hyperbolic geometry, see also \cite{par}.
\subsection{The unit ball model of complex hyperbolic space} 
Let $\mathbb{C}^{2, 1}$ represent the complex $3$-dimensional vector space equipped with the Hermitian form
\begin{equation}\label{the first form}
    \langle \mathbf{z}, \mathbf{w} \rangle = z_1 \overline{w}_1 + z_2 \overline{w}_2 - z_3 \overline{w}_3.
\end{equation}
A vector $\mathbf{z} \in \mathbb{C}^{2, 1} \setminus \left\{ 0\right\}$ is called \emph{negative}, \emph{null}, or \emph{positive} if $\langle \mathbf{z}, \mathbf{z} \rangle < 0$, $\langle \mathbf{z}, \mathbf{z} \rangle = 0$, or $\langle \mathbf{z}, \mathbf{z} \rangle > 0$, respectively. In addition, a positive (resp. negative) vector $\mathbf{z}$ is said to be \emph{normalized} if it satisfies $\langle \mathbf{z}, \mathbf{z} \rangle = 1$ (resp. $\langle \mathbf{z}, \mathbf{z} \rangle = -1$).
Let $\mathbb{P}$ denote the projection map.  
The complex hyperbolic space $\mathbf{H}_{\mathbb{C}}^2$ and its boundary $\partial \mathbf{H}_{\mathbb{C}}^2$ are defined as $\mathbb{P}(V_-)$ and $\mathbb{P}(V_0)$, respectively, where
\[
V_- = \left\{ \mathbf{z} \in \mathbb{C}^{2, 1} \setminus \{ 0 \} : \langle \mathbf{z}, \mathbf{z} \rangle < 0 \right\}, \quad
V_0 = \left\{ \mathbf{z} \in \mathbb{C}^{2, 1} \setminus \{ 0 \} : \langle \mathbf{z}, \mathbf{z} \rangle = 0 \right\}.
\]
We denote $z = \mathbb{P}(\mathbf{z})\in \mathbf{H}_{\mathbb{C}}^2$. With slight abuse of notation, in what follows, we do not distinguish between $z$ and $\mathbf{z}$. 
The complex hyperbolic space is endowed with the Bergman metric:
\[
\cosh^2\left( \frac{\rho(z, w)}{2} \right) = \frac{\langle \mathbf{z}, \mathbf{w} \rangle \langle \mathbf{w}, \mathbf{z} \rangle}{\langle \mathbf{z}, \mathbf{z} \rangle \langle \mathbf{w}, \mathbf{w} \rangle}.
\]

Let $\mathrm{PU}(2, 1)$ denote the holomorphic isometry group of $\mathbf{H}_{\mathbb{C}}^2$, consisting of all elements of $\mathrm{GL}(3, \mathbb{C})$ that preserve the Hermitian form. The elements of $\mathrm{PU}(2, 1)$ are classified into three types, based on their fixed points. Specifically, an isometry $g \in \mathrm{PU}(2, 1)$ is \emph{parabolic} if it has exactly one fixed point on $\partial \mathbf{H}_{\mathbb{C}}^2$, \emph{loxodromic} if it has exactly two distinct fixed points on $\partial \mathbf{H}_{\mathbb{C}}^2$, and \emph{elliptic} otherwise. Moreover, the parabolic elements in $\mathrm{PU}(2, 1)$ are further divided into two types: \emph{unipotent parabolic} and \emph{screw parabolic}. A parabolic element is unipotent if it has three equal eigenvalues; otherwise, it is screw parabolic.

\subsection{Complex hyperbolic ultra-parallel triangle groups}

A \emph{complex line} is the image under the projection $\mathbb{P}$ of a two‑dimensional complex vector subspace of $\mathbb{C}^{2,1}$. A \emph{complex geodesic} is the intersection of a complex line with $\mathbf{H}_{\mathbb{C}}^2$, provided this intersection is non‑empty.

The \emph{Hermitian cross-product} is defined as
$$
z \boxtimes w =
\begin{bmatrix}
\overline{z}_2 \overline{w}_3 - \overline{z}_3 \overline{w}_2 \\
\overline{z}_3 \overline{w}_1 - \overline{z}_1 \overline{w}_3 \\
\overline{z}_2 \overline{w}_1 - \overline{z}_1 \overline{w}_2
\end{bmatrix}.
$$
 A non-zero positive vector $ n \in \mathbb{C}^{2, 1} \setminus \{0\} $ is called a \emph{polar vector} of a complex geodesic $C$ if it satisfies $ \langle n, z \rangle = 0 $ for all $ z \in C $. A polar vector is said to be normalized if $ \langle n, n \rangle = 1 $. Conversely, given a positive vector $n$, the intersection of $\mathbb{P}(n^\perp)$ and $\mathbf{H}_{\mathbb{C}}^2$ is also a complex geodesic. Given two distinct complex geodesics $C_1$ and $C_2$ with corresponding normalized polar vectors $n_1$ and $n_2$, respectively, we define $n_{12} = n_1 \boxtimes n_2$. The complex geodesics $C_1$ and $C_2$ are called \emph{ultra-parallel} if they do not intersect in $\mathbf{H}_{\mathbb{C}}^2 \cup \partial\mathbf{H}_{\mathbb{C}}^2$, and \emph{asymptotic} if their intersection lies on $\partial \mathbf{H}_{\mathbb{C}}^2$. Equivalently, they are ultra-parallel (resp. asymptotic) if and only if $n_{12}$ is positive (resp. null). Furthermore, for two ultra-parallel complex geodesics $C_1$ and $C_2$, the distance between them is given by
$$
\cosh\left(\frac{\rho(C_1, C_2)}{2}\right) = |\langle n_1, n_2 \rangle|.
$$

For a complex geodesic $C$ with polar vector $n$, the \emph{complex reflection} with mirror $C$ is defined as
$$
R_C(z) =\mathbb{P}\left(-\mathbf{z}+2\frac{\langle \mathbf{z}, \mathbf{n} \rangle}{\langle \mathbf{n}, \mathbf{n} \rangle}\mathbf{n} \right),
$$
which acts isometrically on the complex hyperbolic space. The set of all fixed points of $R_C$ is the complex geodesic $C$, and therefore, the isometry $R_C$ is elliptic.

Given three complex reflections $R_1$, $R_2$, and $R_3$ with distinct mirrors $C_1$, $C_2$, and $C_3$, if any pair of these mirrors are ultra-parallel or asymptotic, then we call the group generated by $R_1$, $R_2$, and $R_3$ the \emph{complex hyperbolic ultra-parallel $[m_1, m_2, m_3]$-triangle group}, where $m_j = \rho(C_{j-1}, C_{j+1})$ for $j = 1, 2, 3 \pmod{3}$, and, without loss of generality, we assume that $m_1 \ge m_2 \ge m_3\ge 0$. In Section 3, we provide the parameterization of such a group.

\subsection{Bisectors}
Given two distinct points $p$ and $q$ in $\mathbf{H}_{\mathbb{C}}^2$, the \emph{bisector} $\mathcal{B}(p, q)$ is defined as the set  
$$  
\mathcal{B}(p, q) = \left\{ z \in \mathbf{H}_{\mathbb{C}}^2 : \rho(z, p) = \rho(z, q) \right\}.  
$$  
The complex geodesic spanned by $p$ and $q$ is called the \emph{complex spine}, while the intersection of the complex spine with the bisector is called the  \emph{real spine}. A real spine is a geodesic line in $\mathbf{H}_{\mathbb{C}}^2$, and any  geodesic line uniquely determines a bisector.  

Suppose that $C$ is a complex geodesic with polar vector $n$. The orthogonal projection onto this complex geodesic $C$ is given by  
$$  
\Pi_C(z) = \mathbb{P}\left(\mathbf{z} - \frac{\langle \mathbf{z}, \mathbf{n} \rangle}{\langle \mathbf{n}, \mathbf{n} \rangle}\mathbf{n} \right).  
$$  
The bisector admits a $\mathbb{C}$-slice decomposition. Specifically, let $\sigma$ and $\Sigma$ denote its real spine and complex spine, respectively, and let $\Pi_\Sigma$ denote the orthogonal projection. The bisector can then be expressed as  
$$  
\mathcal{B}(p, q) = \bigcup_{z \in \sigma} \Pi_\Sigma^{-1}(z).  
$$  
For each $z \in \sigma$, the set $\Pi_\Sigma^{-1}(z)$ is a complex geodesic, referred to as a $\mathbb{C}$-slice of the bisector.  

A bisector divides $\mathbf{H}_{\mathbb{C}}^2$ into two components. For a complex reflection, if its mirror is a $\mathbb{C}$-slice, it preserves the bisector and interchanges these two components. Consequently, each component can serve as a fundamental domain for the group generated by such a reflection.  

\subsection{The Siegel domain model of complex hyperbolic space}
Replacing the Hermitian form \eqref{the first form} in Subsection $2.1$ with  
$$
\langle z, w \rangle = z_1\overline{w}_3 + z_2\overline{w}_2 + z_3\overline{w}_1,
$$
we obtain that $z \in \mathbf{H}_{\mathbb{C}}^2$ if $\langle \mathbf{z}, \mathbf{z} \rangle < 0$, and $z \in \partial \mathbf{H}_{\mathbb{C}}^2$ if $\langle \mathbf{z}, \mathbf{z} \rangle = 0$. This defines the \emph{Siegel domain model} of complex hyperbolic space, which we utilize exclusively in Section 6.  

Apart from the difference in the formula for the Hermitian cross-product compared to the unit ball model, all other concepts introduced in Subsections $2.1$–$2.3$ remain applicable to the Siegel domain model, differing only in the choice of the Hermitian form.

The Heisenberg group $\mathcal{H}=\mathbb{C}\times \mathbb{R}$ is equipped with the group law
$$
[\zeta_1, t_1]\ast[\zeta_2, t_2]=[\zeta_1+\zeta_2, t_1+t_2+2\Im(\zeta_1\overline{\zeta}_2)].
$$
Then, $\partial\mathbf{H}_{\mathbb{C}}^2$ can be identified with $\mathcal{H}\cup \left\{ \infty \right\}$ by
$$
\begin{bmatrix}
-|\zeta|^2+it \\
\sqrt{2}\,\zeta \\
1
\end{bmatrix}\sim
[\zeta, t]
,\qquad\text{and}\qquad
\begin{bmatrix}
1 \\
0 \\
0
\end{bmatrix}\sim
\infty.
$$
If a unipotent parabolic isometry fixes the point $\infty$, then it can be written as
$$
T_{[\zeta, t]}=
\begin{bmatrix}
 1 & -\sqrt{2}\,\overline{\zeta} & -|\zeta|^2+it\\
 0 & 1 & \sqrt{2}\,\zeta\\
 0 & 0 & 1
\end{bmatrix}.
$$
Moreover, the isometry $T_{[\zeta, t]}$ acts on $\mathcal{H}\cup \left\{ \infty \right\}$ as the left translation by $[\zeta, t]$. We also refer to $T_{[\zeta, t]}$ as the Heisenberg translation by $[\zeta, t]$.

The boundary of a complex geodesic is called an \emph{infinite chain} (or \emph{finite chain}) if it passes (or does not pass) through the point $\infty$. For any infinite chain, its image under the \emph{vertical projection} $\Pi$, given by $[\zeta, t] \mapsto [\zeta, 0]$, is a point on the plane $\mathbb{C} \times \{0\}$. Conversely, for a point $[\zeta, 0]$, its preimage under $\Pi$ is an infinite chain. Thus, an infinite chain can be uniquely determined by a complex number, and we often use $\zeta = \zeta_0$ to denote the infinite chain whose image under the vertical projection is the point $[\zeta_0, 0]$. According to \cite{mpp}, or by direct computation, the polar vector corresponding to a given infinite chain $\zeta = \zeta_0$ is $(-\sqrt{2}\,\overline{\zeta}_0, 1, 0)^T$.

The Heisenberg group is equipped with the \emph{Cygan metric}, which is given by
$$
\rho_0\left([\zeta_1, t_1], [\zeta_2, t_2]\right)=\left||\zeta_1-\zeta_2|^2-it_1+it_2-2i\Im(\zeta_1\overline{\zeta}_2) \right|^{1/2},
$$
and the \emph{Cygan sphere} of radius $r>0$ and center $[\zeta_0, t_0]$ is defined by
$$
\left\{[\zeta, t]\in \mathcal{H} : \rho_0([\zeta, t], [\zeta_0, t_0])=r \right\}.
$$
For a unit circle in the plane $\mathbb{C} \times \{0\}$ that is also a finite chain, the corresponding complex reflection can be expressed as
$$
\begin{bmatrix}
0  & 0 & 1 \\
0  & -1 & 0\\
1  & 0 & 0
\end{bmatrix}.
$$
The Cygan sphere of radius $1$ and center $o = [0, 0]$ is invariant under this complex reflection. Moreover, this reflection maps the exterior of the unit Cygan sphere to the interior of the unit Cygan sphere. 

\medskip

%Lemma~\ref{Klein} is a variant of Klein's combination theorem, as used in \cite{sch1}. The result in \cite{lx} proved this lemma using a method similar to that of Lemma $2.2.1.4$ in \cite{wg}. Moreover, Lemma~\ref{Klein} can be regarded as a corollary of Theorem $\mathrm{A}.10$ in \cite{mas}. Our main results can also be derived using Klein's combination theorem provided in \cite{sch1}. 

\section{Parameterization}\label{sec para}
In this section, we describe the parameter space of the complex hyperbolic ultra-parallel $[m_1, m_2, m_3]$-triangle groups with $m_3>0$. 
Let $C_j$ for $j = 1, 2, 3$ be three distinct complex geodesics with normalized polar vectors $n_j$, respectively. The distance $m_j$ between two complex geodesics $C_{j-1}$ and $C_{j+1}$ is given by 
$$
\cosh\left(\frac{m_j}{2}\right)=|\langle n_{j-1}, n_{j+1}\rangle|.
$$
For convenience, we denote $r_j=\cosh\left(\frac{m_j}{2}\right)$ and $ s_j =\sinh\left(\frac{m_j}{2}\right)$ for $j=1, 2$ and $3$.
The proof of the following result is analogous to Proposition 1 in \cite{pra}.

\begin{prop}\label{angular}
A complex hyperbolic ultra-parallel $[m_1, m_2, m_3]$-triangle group with $m_3 > 0$ exists if and only if 
\begin{equation}\label{non-degenerate}
    \cos\alpha \le \frac{r_1^2 + r_2^2 + r_3^2 - 1}{2r_1 r_2 r_3},
\end{equation}
where 
$$
\alpha = \arg\left(\langle n_1, n_3 \rangle \langle n_2, n_1 \rangle \langle n_3, n_2 \rangle\right) \in [0, 2\pi].
$$
Moreover, if such an ultra‑parallel $[m_1, m_2, m_3]$-triangle group exists, then, up to the action of $\mathrm{PU}(2,1)$, the corresponding polar vectors of the mirrors $C_1$, $C_2$, and $C_3$ can be normalized as
\begin{equation}\label{parameterization}
  n_1 = \begin{pmatrix}
 0 \\
 r_3  \\
 s_3
\end{pmatrix}, \quad
n_2 = \begin{pmatrix}
 0 \\
 1  \\
 0
\end{pmatrix}, \quad
n_3 = \begin{pmatrix}
 \sqrt{|s_3^{-1}(r_1 r_3 - r_2 e^{-i \alpha})|^2 - r_1^2 + 1} \\
 r_1 \\
 s_3^{-1}(r_1 r_3 - r_2 e^{-i \alpha}) 
\end{pmatrix}.
\end{equation}
\end{prop}
\begin{proof}
    After normalization, we obtain the three normalized polar vectors
\[
n_1=\begin{pmatrix}
 0 \\
 \delta  \\
\gamma
\end{pmatrix}, \quad
n_2=\begin{pmatrix}
 0 \\
 1  \\
0
\end{pmatrix},\quad
n_3=\begin{pmatrix}
 z\\
 \xi \\
\beta 
\end{pmatrix},
\]
where $\beta \in \mathbb{C}$, and $\xi, z, \gamma, \delta \in \mathbb{R}^+$. There are six identities:

\begin{enumerate}
\item[(i)] $|\langle n_3, n_2 \rangle| = \xi = r_1$;
\item[(ii)] $|\langle n_1, n_3 \rangle| = |\delta \xi - \gamma \bar{\beta}| = r_2$;
\item[(iii)] $|\langle n_2, n_1 \rangle| = \delta = r_3$;
\item[(iv)] $\langle n_3, n_3 \rangle = z^2 + \xi^2 - |\beta|^2 = 1$;
\item[(v)] $\langle n_1, n_1 \rangle = \delta^2 - \gamma^2 = 1$;
\item[(vi)] $\arg(\langle n_1, n_3 \rangle \langle n_2, n_1 \rangle \langle n_3, n_2 \rangle) = \alpha$.
\end{enumerate}

It follows from identities (i), (iii), and (v) that we have $\xi = r_1$, $\delta = r_3$ and $\gamma = s_3$.
The condition $m_3 > 0$ implies $s_3 > 0$.
From (ii) and (vi), we have
$$
\delta \xi - \gamma \bar{\beta} = r_2 e^{i\alpha},
$$
where $\alpha \in [0, 2\pi]$. This is equivalent to
$$
\beta = s_3^{-1} (r_1 r_3 - r_2 e^{-i \alpha}).
$$
Then, using identity (iv), we get
$$
z = \sqrt{|s_3^{-1} (r_1 r_3 - r_2 e^{-i \alpha})|^2 - r_1^2 + 1}.
$$
Note that the following condition must be satisfied:
$$
z^2 = |s_3^{-1} (r_1 r_3 - r_2 e^{-i \alpha})|^2 - r_1^2 + 1 \ge 0.
$$
This is equivalent to
$$
\cos\alpha \le \frac{r_1^2 + r_2^2 + r_3^2 - 1}{2r_1 r_2 r_3}.
$$
\end{proof}

\begin{rem}\label{remk angular}
    Following \cite{pra}, the parameter $\alpha$ is known as the angular invariant of the triangle $(C_1, C_2, C_3)$. As noted in the remark following Proposition~1 in Section~3 of \cite{pra}, we assume throughout the paper that $\alpha \in [0, \pi]$.
\end{rem}

Under the normalization as presented in \eqref{parameterization}, a straightforward calculation shows that
\begin{equation}\label{R2R1}
\begin{aligned}
R_1&=\begin{bmatrix}
-1  & 0 & 0\\
0  & r_3^2+s_3^2 & -2r_3s_3 \\
0  & 2r_3s_3 & -r_3^2-s_3^2
\end{bmatrix}, 
& R_2&=\begin{bmatrix}
-1  & 0 & 0\\
0  & 1 & 0\\
0  & 0 & -1
\end{bmatrix},\\
R_3&=\begin{bmatrix}
2z_1^2-1  & 2z_1r_1 & -2z_1\overline{z}_3\\
2z_1r_1 & 2r_1^2-1 & -2\overline{z}_3r_1\\
2z_1z_3 & 2r_1z_3 & -2|z_3|^2-1
\end{bmatrix},
& R_2R_1&=\begin{bmatrix}
1  & 0 & 0\\
0  & r_3^2+s_3^2 & -2r_3s_3 \\
0  & -2r_3s_3 & r_3^2+s_3^2
\end{bmatrix}.
\end{aligned}
\end{equation}
Here, $z_1$ and $z_3$ represent the first and last components of the polar vector $n_3$, respectively. Specifically, we have
$$
z_1^2=\frac{r_1^2+r_2^2+r_3^2-2r_1r_2r_3\cos\alpha-1}{r_3^2-1},\qquad
|z_3|^2=\frac{r_2^2+r_1^2r_3^2-2r_1r_2r_3\cos\alpha}{r_3^2-1}.
$$ 
and $z_3=(r_1r_3-r_2e^{-i\alpha})/s_3$.
Denote by $C_{12}$ the common orthogonal complex geodesic to $C_1$ and $C_2$.  Under the normalization given in \eqref{parameterization}, we can now express its polar vector as $n_{12}=(1, 0, 0)^T$.
The complex geodesics $C_3$ and $C_{12}$ are ultra-parallel if and only if
$$
\frac{\langle n_3, n_{12} \rangle\langle n_{12}, n_3 \rangle}{\langle n_3, n_3 \rangle\langle n_{12}, n_{12} \rangle}=|\langle n_3, n_{12} \rangle|^2>1,
$$
which is equivalent to
\begin{equation}\label{condition on cos alpha}
  \cos\alpha<\frac{r_1^2+r_2^2}{2r_1r_2r_3}. 
\end{equation}
In the following,  we assume that $C_3$ and $C_{12}$ are ultra-parallel.
Suppose that $C_0$ is the common orthogonal complex geodesic to $C_{3}$ and $C_{12}$. Then, a polar vector of $C_0$ is
$$
n_0=s_3(n_3\boxtimes n_{12})=
\begin{pmatrix}
0\\
 r_1r_3-r_2e^{i\,\alpha} \\
s_3r_1
\end{pmatrix}.
$$
Therefore, the intersection of $C_0$ and $C_{12}$ is 
\begin{equation}\label{p3 point 1}
    p_3=n_0\boxtimes n_{12}=
\begin{pmatrix}
0\\
 s_3r_1 \\
r_1r_3-r_2e^{-i\,\alpha} 
\end{pmatrix},
\end{equation}
and the distance $d$ between $C_{12}$ and $C_3$ is given by
$$
\cosh\left(\frac{d}{2}\right)=\sqrt{|s_3^{-1}(r_1 r_3 - r_2 e^{-i \alpha})|^2 - r_1^2 + 1}=\sqrt{M},
$$
where 
\begin{equation}\label{value of M}
    M=\frac{r_1^2+r_2^2+r_3^2-2r_1r_2r_3\cos\alpha-1}{r_3^2-1}.
\end{equation}
Since the complex geodesic $C_{12}$ can be identified with the unit disc in the complex plane $\mathbb{C}$, with slight abuse of notation, we also write 
\begin{equation}\label{p3 point 2}
    p_3=\frac{s_3r_1}{r_1r_3-r_2e^{-i\,\alpha}},
\end{equation}
when working within the complex geodesic $C_{12}$. 
We have
\begin{equation}\label{Re(p_3)}
    \Re(p_3)= \frac{s_3 r_1 (r_1 r_3 - r_2 \cos \alpha)}{(r_1 r_3 - r_2 \cos \alpha)^2 + r_2^2 \sin^2 \alpha},\quad
    \Im(p_3)=\frac{-s_3 r_1 r_2 \sin \alpha}{(r_1 r_3 - r_2 \cos \alpha)^2 + r_2^2 \sin^2 \alpha}.
\end{equation}
It is evident that, under condition~\eqref{condition on cos alpha} and the assumption \( m_1 \ge m_2 \ge m_3 > 0 \), we have \( \Re(p_3) > 0 \) and \( \Im(p_3) \le 0 \).
The latter inequality holds since the angular invariant \( \alpha \) is restricted to the interval \( [0, \pi] \) (see Remark~\ref{remk angular}).

The orthogonal projection of $C_j$ onto $C_{12}$ is a single point, denoted by $p_j$ for $j = 1, 2$. Moreover, we have \( p_1 = ( 0,\, s_3/r_3,\, 1 )^T \) and \( p_2 = ( 0,\, 0,\, 1 )^T \). With a slight abuse of notation, we also write these as $p_1 = s_3/r_3$ and $p_2 = 0$ when working within the complex geodesic $C_{12}$.
\medskip

We now consider the existence of a bisector with \( C_3 \) as its slice, whose orthogonal projection onto the complex geodesic \( C_{12} \) is an open disc, under condition \eqref{condition on cos alpha}. This result will be used in the next section.

\begin{lem}[Section~3.2.2 of \cite{gol}]\label{ultra parallel distance}
    Suppose that \( H_1 \) and \( H_2 \) are two ultra-parallel complex geodesics, and let \( H_0 \) denote the common orthogonal complex geodesic to \( H_1 \) and \( H_2 \). Let \( x = H_0 \cap H_1 \). Then, \( \Pi_1 \) maps \( H_2 \) diffeomorphically onto the geometric ball centered at \( x \) with radius
    \[
    2\tanh^{-1} \sech\left(\frac{\rho(H_1, H_2)}{2}\right),
    \]
    where \( \Pi_1 \) is the orthogonal projection onto \( H_1 \).
\end{lem}

By using Lemma~\ref{ultra parallel distance} we obtain:
\begin{lem}\label{projection onto C23}
    Suppose that the complex geodesics $C_3$ and $C_{12}$ are ultra-parallel, and let $C_0$ denote the common orthogonal complex geodesic to $C_3$ and $C_{12}$. The distance between $C_3$ and $C_{12}$ is given by 
    $$
    d=\rho(C_3, C_{12})=2\cosh^{-1}(|\langle n_3, n_{12} \rangle|)=2\cosh^{-1}(\sqrt{M}),
    $$
    where $M$ is given in equation~\eqref{value of M}.
    Additionally, $\Pi_{12}$ maps $C_3$ onto an open disc (in the hyperbolic metric) centered at $p_3 = C_0 \cap C_{12}$, with radius $d_3$ determined by the equation 
    \begin{equation}\label{d3}
        \cosh(d_3) = \frac{M + 1}{M - 1},
    \end{equation}
    where $\Pi_{12}$ denotes the orthogonal projection onto $C_{12}$.
\end{lem}

\begin{proof}
   From Lemma~\ref{ultra parallel distance}, we have
   \[
   d_3 = 2 \tanh^{-1} \sech\left(\frac{d}{2}\right) = 2 \tanh^{-1}\left(\frac{1}{\sqrt{M}}\right).
   \]
   Solving the system of equations
    \[
    \frac{\sinh(d_3/2)}{\cosh(d_3/2)} = \frac{1}{\sqrt{M}}, \qquad
    \cosh^2\left(\frac{d_3}{2}\right) - \sinh^2\left(\frac{d_3}{2}\right) = 1,
    \]
    we obtain
   \[
   \cosh^2\left(\frac{d_3}{2}\right) = \frac{M}{M - 1}.
   \]
   This leads to
   \[
   \cosh(d_3) = \frac{M + 1}{M - 1}.
    \]
\end{proof}

\begin{lem}\label{C3 existence}
    If the complex geodesics \( C_3 \) and \( C_{12} \) are disjoint, then there exists a bisector \( \mathcal{B}(C_3) \) with \( C_3 \) as its slice such that  
    \( \Pi_{12}(\mathcal{B}(C_3)) = \Pi_{12}(C_3) = S_3 \),  
    where \( S_3 \) is the open disc in \( C_{12} \) of radius \( d_3 \), centered at \( p_3 \), with \( d_3 \) given by equation~\eqref{d3}.
\end{lem}

\begin{proof}
The existence of \( \mathcal{B}(C_3) \) follows from the argument presented in the proof of Lemma~9.1 of \cite{bk}. For the reader’s convenience, however, we briefly outline the method used to construct such a bisector.
Let \( C_0 \) denote the common orthogonal complex geodesic to \( C_3 \) and \( C_{12} \). The intersections \( C_3 \cap C_0 \) and \( C_{12} \cap C_0 \) can be connected by a geodesic segment, denoted by $\sigma$, lying within \( C_0 \). The desired bisector \( \mathcal{B}(C_3) \) is the one whose real spine is the geodesic line in \( C_0 \) that is orthogonal to \( \sigma \) at the point \( C_3 \cap C_0 \).

Note that the bisector \( \mathcal{B}(C_3) \) admits a slice decomposition. The distance between \( C_{12} \) and any slice other than \( C_3 \) is greater than the distance between \( C_{12} \) and \( C_3 \). Consequently, the orthogonal projections of these other slices onto \( C_{12} \) are also discs centered at \( p_3 \), but with smaller radii than \( d_3 \). Therefore, the orthogonal projection of \( \mathcal{B}(C_3) \) onto \( C_{12} \) is precisely \( S_3 \).
\end{proof}

\medskip

\section{Discreteness result}\label{sec discreteness result}
In this section, we prove a discreteness result, that is Proposition~\ref{proposition}, of complex hyperbolic ultra-parallel $[m_1, m_2, m_3]$-triangle groups with $m_3>0$. This result is used to prove Theorem~\ref{main thm 2} in Section $5$.
We first present a version of Klein's combination theorem, together with a preliminary discreteness result in Subsection~\ref{sketch}, followed by the complete proof in Subsection~\ref{the proof}.
Then, in Subsection~\ref{algorithm}, we offer a more detailed discussion of condition~(3) in Proposition~\ref{proposition}. 

\medskip

\subsection{A preliminary discreteness result}\label{sketch}
The following variant of Klein's combination theorem, given in~\cite{lx}, is closely related to the version in~\cite{sch1} and can be regarded as a corollary of Theorem~$\mathrm{VII}.\mathrm{A}.10$ in~\cite{mas}.

\begin{lem}\label{Klein}
 Let $G_{1}$ and $G_{2}$ be subgroups of $\mathrm{PU}(2,1)$ whose union generates the group $G$. If there exist nonempty open subsets $U_1$, $U_2$, $V_1$ and $V_2$ of $\mathbf{H}_{\mathbb{C}}^2$ (or $\partial\mathbf{H}_{\mathbb{C}}^2$) with $U_1 \cap U_2=\emptyset$ and $V_j\subsetneq U_j$ for $j=1, 2$, so that each nontrivial element of $G_1$ maps $U_1$ into $V_2$ and each nontrivial element of $G_2$ maps $U_2$ into $V_1$, then $G$ is discrete. Moreover, if $G_1\cap G_2=\left \{ id \right \}$, then $G$ is the free product of the group $G_1$ and $G_2$.
\end{lem}

\medskip
Based on the standard Klein combination argument, we establish the following lemma.

\begin{lem}\label{lem}
    A complex hyperbolic $[m_1, m_2, m_3]$-triangle group representation with $m_3 > 0$ and angular invariant $\alpha \in [0, \pi]$ is discrete and faithful if the following two conditions are satisfied:  
    \begin{enumerate}
        \item[(1)] $\cos\alpha < \frac{r_1^2+r_2^2}{2r_1r_2r_3}$;
        \item[(2)] the disc $S_3$ does not intersect $(R_2R_1)^n(S_3)$ and $R_1(R_2R_1)^n(S_3)$ for all $n \in \mathbb{Z} \setminus \{ 0 \}$.
    \end{enumerate}
\end{lem}
\begin{proof}
    Condition~(1) in the lemma statement is equivalent to the statement that the two complex geodesics $C_3$ and $C_{12}$ are ultra-parallel, by equation~\eqref{condition on cos alpha}. Let $\Pi_{12}$ denote the orthogonal projection onto the complex geodesic $C_{12}$. Under Condition~(1), Lemma~\ref{C3 existence} asserts that there exists a bisector $\mathcal{B}(C_3)$, with $C_3$ as its slice, such that
$$
\Pi_{12}(\mathcal{B}(C_3)) = \Pi_{12}(C_3) = S_3,
$$
where $S_3$ is an open disc (with respect to the hyperbolic metric) in $C_{12}$, centred at $p_3$ and of radius $d_3$, with $p_3$ and $d_3$ given in equations~\eqref{p3 point 1} and~\eqref{d3}, respectively.

Condition~(2) in the lemma statement is equivalent to the requirement that the open disc $S_3$ does not intersect any of its non-trivial images under the group $G_2 = \langle R_1, R_2 \rangle$. Indeed, the elements of the group $G_2 = \langle R_1, R_2 \rangle$ can be classified into two types:
\(\mathrm{(i)}\) $(R_2 R_1)^n$,
\(\mathrm{(ii)}\) $R_1 (R_2 R_1)^n$,
where $n \in \mathbb{Z}$. Moreover, since the groups under consideration are ultra-parallel triangle groups, the open discs $S_3$ and $R_j(S_3)$ are disjoint for $j = 1, 2$.

Note that the bisector $\mathcal{B}(C_3)$ divides $\mathbf{H}_{\mathbb{C}}^2$ into two components, each of which can serve as a fundamental domain for the group $G_1 = \langle R_3 \rangle$. Let $F_3$ denote the component contained entirely in $\Pi_{12}^{-1}(S_3)$, and let $F_3'$ denote the other component. From these two conditions, we can construct the required four open sets. Define
$$
U_1 = V_1 \cup D, \quad
V_1 = \bigcup_{g \in G_2 \setminus \{\mathrm{id}\}} g(F_3), \quad
U_2 = F_3, \quad
V_2 = R_3(U_1),
$$
where $D$ is any nonempty, proper, open subset of $\mathbf{H}_{\mathbb{C}}^2 \setminus (F_3 \cup V_1)$.
Conditions~(1) and~(2) guarantee that the fundamental domain $F_3$ for $G_1 = \langle R_3 \rangle$ does not intersect any of its non-trivial images under $G_2$, and that $V_1 \subsetneq F_3'$. Clearly, $V_j \subsetneq U_j$ for $j = 1, 2$, and each non-trivial element of $G_1$ (resp.\ $G_2$) maps $U_1$ (resp.\ $U_2$) into $V_2$ (resp.\ $V_1$). This implies that these four open sets satisfy the conditions stated in Lemma~\ref{Klein}, thereby completing the proof of the lemma.
\end{proof}
\medskip
It is not difficult to see that condition~(1) in Lemma~\ref{lem} implies inequality~\eqref{non-degenerate} in Proposition~\ref{angular}.
\medskip

\subsection{A more precise formulation of condition~$(2)$ in Lemma~\ref{lem}}\label{the proof} 

We begin by giving a necessary and sufficient condition for the separation of the discs $S_3$ and $(R_2R_1)^n(S_3)$ for all $n \in \mathbb{Z}\setminus\{0\}$.
\begin{lem}\label{lem 4.3}
    The condition that the open disc $S_3$ does not intersect $(R_2R_1)^n(S_3)$ for all $n \in \mathbb{Z} \setminus \left\{0\right\}$ is equivalent to
\[
a_2 \cos^2 \alpha + a_1 \cos \alpha + a_0 \geq 0,
\]
where the coefficients $a_j$ for $j = 0, 1, 2$ are defined by
\begin{equation}
a_0 = 4r_1^2r_2^2r_3^4 + r_3^2(r_1^2 - r_2^2)^2 - r_1^2 - r_2^2 - r_3^2 + 1,
\label{eq:a0}
\end{equation}

\begin{equation}
a_1 = 2r_1r_2r_3\bigl(1 - 2r_3^2(r_1^2 + r_2^2)\bigr),
\qquad
a_2 = 4r_1^2r_2^2r_3^2.
\label{eq:a1a2}
\end{equation}
\end{lem}

\begin{proof}
Recall that \( S_3 \) is the open disc in the complex geodesic \( C_{12} \), centred at \( p_3 \) with radius \( d_3 \), where \( d_3 \) is given by Lemma~\ref{projection onto C23}.
Specifically, we have
$$
p_3 = 
\begin{pmatrix}
0 \\
s_3r_1 \\
r_1r_3 - r_2 e^{-i \alpha}
\end{pmatrix}, \qquad
\cosh(d_3) = \frac{M + 1}{M - 1},
$$
where $M$ is given in equation~\eqref{value of M}.
Since \(R_2R_1\) (given by equation~\eqref{R2R1}), restricted to \(C_{12}\), is a hyperbolic element, we have that \((R_2R_1)^n(S_3)\) and \(S_3\), for \(n \in \mathbb{Z} \setminus \{0\}\), are disjoint if and only if \(S_3\) and \(R_2R_1(S_3)\) are disjoint. This condition is equivalent to
$$
\cosh^2 \left( \frac{\rho(R_2R_1(p_3), p_3)}{2} \right) - \cosh^2 \left( \frac{2d_3}{2} \right) \geq 0.
$$
After simplification this condition reduces to
\[
a_2 \cos^2 \alpha + a_1 \cos \alpha + a_0 \geq 0,
\]
where the coefficients $a_j$ for $j = 0, 1, 2$ are defined in~\eqref{eq:a0} and~\eqref{eq:a1a2}.
\end{proof}

\medskip

We then give an equivalent statement for the separation of the discs $S_3$ and $R_1 (R_2 R_1)^n (S_3)$ for all $n \in \mathbb{Z} \setminus \{0\}$.

\begin{lem}\label{lem 4.4}
    The following are equivalent:
\begin{enumerate}
    \item[(A)] For all $n \in \mathbb{Z} \setminus \{0\}$, $S_3$ does not intersect $R_1 (R_2 R_1)^n(S_3)$.
    \item[(B)] Both of
    \begin{enumerate}
        \item[(1)] For all $n \in \mathbb{Z}^+$, $C_1$ does not intersect $(R_2 R_1)^n(S_3)$, and
        \item[(2)] For all $n \in \mathbb{Z}^+$, $C_2$ does not intersect $(R_2 R_1)^n(S_3)$
    \end{enumerate}
    hold.
\end{enumerate}
\end{lem}
\begin{proof}
    Since each $R_j$ is a complex reflection with the mirror $C_j$ for $j=1, 2$, we have that
    \begin{equation*}
    \begin{aligned}
        (R_2R_1)^{n}(S_3)\cap C_1=\emptyset 
        &\Leftrightarrow (R_2R_1)^{n}(S_3)\cap R_1 (R_2R_1)^{n}(S_3)=\emptyset\\
        &\Leftrightarrow S_3\cap (R_1R_2)^n R_1 (R_2R_1)^{n}(S_3)=\emptyset\\
        &\Leftrightarrow S_3\cap R_1 (R_2R_1)^{2n}(S_3)=\emptyset,
    \end{aligned}
\end{equation*}
and
\begin{equation*}
    \begin{aligned}
        (R_2R_1)^{n}(S_3)\cap C_2=\emptyset 
        &\Leftrightarrow (R_2R_1)^{n}(S_3)\cap R_2 (R_2R_1)^{n}(S_3)=\emptyset\\
        &\Leftrightarrow S_3\cap (R_1R_2)^n R_2 (R_2R_1)^{n}(S_3)=\emptyset\\
        &\Leftrightarrow S_3\cap R_1 (R_2R_1)^{2n-1}(S_3)=\emptyset,
    \end{aligned}
\end{equation*}
   for $n\in \mathbb{Z}$. Note that the statement that the disc $S_3$ does not intersect $R_1(R_2R_1)^{-1}(S_3)$ holds, because, from above, this is equivalent to the fact that $S_3$ does not intersect $C_2$. 
    Hence, to prove this lemma, it suffices to show that each $(R_2R_1)^{-m}(S_3)$ does not intersect the complex geodesics $C_1$ and $C_2$ for all $m \in \mathbb{Z}^+$.
    
    The complex geodesic $C_{12}$ can be identified with $\mathbf{H}_{\mathbb{C}}^1$. By equation~\eqref{Re(p_3)}, we have $\Re(p_3) > 0$ and $\Im(p_3) \leq 0$. There are three possible cases for the position of the open disc $S_3$:  
    \begin{itemize}
    \item[(i)] The open disc $S_3$ does not intersect the axis $(-1,1)$.  
    \item[(ii)] $S_3 \cap (-1,1) \neq \emptyset$, and the orthogonal projection of $p_3$ onto $(-1,1)$  lies in the interval $[p_2, p_1]$, where $p_j = \Pi_{12}(C_j)$ for $j = 1,2$. Specifically, \( p_1 = s_3/r_3 > 0 \) and \( p_2 = 0 \). Here, $(-1,1)$ denotes the geodesic line with endpoints $-1$ and $1$.  
    \item[(iii)] $S_3 \cap (-1,1) \neq \emptyset$, and the orthogonal projection of $p_3$  onto $(-1,1)$ lies in the geodesic ray $(p_1,1)$, excluding the endpoint $p_1$.  
\end{itemize}
See Figure~\ref{pic}. 

%\begin{figure}[htbp]
%    \centering
%    \includegraphics[scale=0.8]{pic.png}
%    \caption{}
%    \label{pic}
%\end{figure}

\begin{figure}[htbp]
    \centering
    \begin{tikzpicture}[scale=0.8]
    % Scale factor: 6/1.6 = 3.75
    \def\scale{3.75}
    
    % Draw 140° to 220° arc (left part) - scaled to 6cm
    \draw[very thick] (140:6) arc[start angle=140, end angle=220, radius=6];
    
    % Draw -50° to 50° arc (right part) - scaled to 6cm
    \draw[very thick] (-40:6) arc[start angle=-40, end angle=40, radius=6];
    
    % Calculate coordinates of four endpoints - scaled to 6cm
    \coordinate (A) at (140:6);  % Upper left point
    \coordinate (B) at (40:6);   % Upper right point
    \coordinate (C) at (220:6);  % Lower left point
    \coordinate (D) at (-40:6);  % Lower right point
    
    % Draw zigzag lines - scaled proportionally with thicker lines
     \draw[thick, decorate, decoration={zigzag, segment length=7.5mm, amplitude=0.9375mm}] (A) -- (B);
    \draw[thick, decorate, decoration={zigzag, segment length=7.5mm, amplitude=0.9375mm}] (C) -- (D);
    
    % Draw diameter - scaled to 6cm with thicker line
    \draw[very thick] (-6,0) -- (6,0);
    
    % Mark points
    \node[left] at (-6,0) {$-1$};
    \node[right] at (6,0) {$1$};
    
\draw[fill] (0,0) circle [radius=0.1];
\draw[fill] (2,0) circle [radius=0.1];
\node [above] at (0,0.05) {$p_2$};
\node [above] at (2,0.05) {$p_1$};
\draw [very thick] (1,-1) circle [radius=1.2]; 
\draw[fill] (1,-1) circle [radius=0.1];
\draw [very thick] (4,-0.6) circle [radius=0.8];
\draw[fill] (4,-0.6) circle [radius=0.1];

\draw [very thick] (3,-2.5) circle [radius=0.7];
\draw[fill] (3,-2.5) circle [radius=0.1];

\draw [ultra thick,->] (-2,2)-- (2,2);
\node [above] at (0,2) {$R_1R_2$};
\node at (5,-1.5) {case(iii)};
\node at (-1,-1) {case(ii)};
\node at (1.7,-3) {case(i)};

\draw[very thick,dashed](1,-1)--(1,0);
\draw[very thick,dashed](4,-0.6)--(4,0);
\end{tikzpicture}
\caption{}
\label{pic}
\end{figure}

In all cases, the open disc $S_3$ does not intersect the complex geodesics $C_1$ and $C_2$ because the groups we consider are ultra-parallel triangle groups. According to Lemma~\ref{translation length}, the isometry $R_1R_2$ restricted to the complex geodesic $C_{12}$ is hyperbolic, with axis $(-1,1)$ and translation length $2m_3$. Moreover, it has $1$ as its attractive fixed point and $-1$ as its repulsive fixed point. 
Note also that the hyperbolic distance between $p_1$ and $p_2$ is $m_3$. Hence, each $(R_2R_1)^{-m}(S_3)$ does not intersect the complex geodesics $C_1$ and $C_2$ for all $m \in \mathbb{Z}^+$.
\end{proof}

\medskip

The following result shows that whether the two discs $S_3$ and $R_1 (R_2 R_1)^n (S_3)$ are disjoint depends entirely on the types of specific isometries.

\begin{lem}\label{lem 4.5}
    Condition~$(1)$ (resp. $(2)$) in Lemma~\ref{lem 4.4} is equivalent to the isometries $R_1(R_2R_1)^{2n}R_3$ (resp. $R_1(R_2R_1)^{2n-1}R_3$) are non-elliptic for all $n\in \mathbb{Z^+}$.
\end{lem}
\begin{proof}
    Note that the open disc $S_3$ is the orthogonal projection of $C_3$ onto $C_{12}$. Moreover, by Proposition~2.5 of \cite{bk}, we have
    $$
    \Pi_{12}\circ (R_2R_1)^n = (R_2R_1)^n \circ \Pi_{12},
    $$
    where $\Pi_{12}$ denotes this orthogonal projection.
    Hence, Lemma~\ref{lem 4.4} still holds when $S_3$ is replaced by $C_3$. Therefore, conditions (1) (resp. (2)) in Lemma~\ref{lem 4.4} are equivalent to the complex geodesics $C_1$ (resp. $C_2$) and $(R_2 R_1)^n(C_3)$ being asymptotic or ultra-parallel. This, in turn, is equivalent to  the isometries $R_1(R_2 R_1)^n R_3(R_1 R_2)^n$ (resp.~$R_2(R_2 R_1)^n R_3(R_1 R_2)^n$) being non-elliptic for all $n \in \mathbb{Z}^+$. It is clear that the isometries $R_1(R_2 R_1)^n R_3(R_1 R_2)^n$ and $(R_1 R_2)^n R_1(R_2 R_1)^n R_3 = R_1(R_2 R_1)^{2n} R_3$ are conjugate. Similarly, the isometries $R_2(R_2 R_1)^n R_3(R_1 R_2)^n$ and $R_1(R_2 R_1)^{2n-1} R_3$ are conjugate.
\end{proof}

%It follows directly from Lemmas~\ref{lem}, \ref{lem 4.3}, \ref{lem %4.4}, and \ref{lem 4.5} that the following result can be %immediately obtained.

Lemmas~\ref{lem}, \ref{lem 4.3}, \ref{lem 4.4}, and \ref{lem 4.5} immediately yield the following result.

\begin{prop}\label{proposition}
    A complex hyperbolic $[m_1, m_2, m_3]$-triangle group representation with $m_3>0$ and angular invariant $\alpha \in [0, \pi]$ is discrete and faithful if the following three conditions are satisfied:
    \begin{enumerate}
        \item[(1)] $\cos\alpha < \frac{r_1^2+r_2^2}{2r_1r_2r_3}$;
        \item[(2)] $a_2\cos^2\alpha + a_1\cos\alpha + a_0 \geq 0$,
where the coefficients $a_j$ for $j = 0, 1, 2$ are defined in equations~\eqref{eq:a0} and~\eqref{eq:a1a2}.
        \item[(3)] the isometries $w^{(n)}=R_1(R_2R_1)^{n}R_3$ are non-elliptic for all $n\in \mathbb Z^+$.
    \end{enumerate}
\end{prop}

In fact, for a given complex hyperbolic $[m_1, m_2, m_3]$-triangle group representation with angular invariant $\alpha$, it suffices to check the isometry types of $w^{(n)}$ for at most two positive integers $n$ in condition~$(3)$ of Proposition~\ref{proposition}.
In Subsection~\ref{algorithm}, we provide an algorithm to determine which two isometries need to be checked.
\medskip

\subsection{An Algorithm}\label{algorithm}
The complex geodesic \( C_{12} \) is identified with \( \mathbf{H}_{\mathbb{C}}^1 \). Let \( (-1, 1) \) denote the geodesic line with endpoints \( -1 \) and \( 1 \). We define \( v_0 \) as the intersection of \( C_2 \) and \( C_{12} \), and for \( k \in \mathbb{Z} \), let \( v_k \in (-1, 1) \) be the point at a distance of \( |k| m_3 \) from \( v_0 \), where \( k > 0 \) implies \( v_k > 0 \) and \( k < 0 \) implies \( v_k < 0 \). In this notation, the intersection of \( C_1 \) and \( C_{12} \) is denoted by \( v_1 \).

Let \( D_j =(v_{j-1}, v_j] \) denote the geodesic segment with the endpoint \( v_{j-1} \) excluded. Recall that \( \Pi_{(-1, 1)} \) represents the orthogonal projection onto the geodesic line \( (-1, 1) \).

%\begin{figure}[htbp]
%       \centering
%        \includegraphics[scale=0.8]{figure_2k+1.png}
%        \caption{Schematic diagram for \( p_3^0 = \Pi_{(-1, 1)}%(p_3) \) in \( D_3 \).}
%         \label{figure 2k+1.png}
%\end{figure}

\begin{figure}[htbp]
\centering
\begin{tikzpicture}[scale=0.8]
    % 定义裁剪区域 - 根据你的图形调整这个矩形
   \clip (-9,-5) rectangle (11,5);  % 调整这4个数字来控制显示范围
    
    % Draw 140° to 220° arc (left part) - scaled to 6cm
    \draw[very thick] (140:6) arc[start angle=140, end angle=220, radius=6];
    
    % Draw -50° to 50° arc (right part) - scaled to 6cm
    \draw[very thick] (-40:6) arc[start angle=-40, end angle=40, radius=6];
    
    % Calculate coordinates of four endpoints - scaled to 6cm
    \coordinate (A) at (140:6);  % Upper left point
    \coordinate (B) at (40:6);   % Upper right point
    \coordinate (C) at (220:6);  % Lower left point
    \coordinate (D) at (-40:6);  % Lower right point
    
    % Draw zigzag lines - scaled proportionally with thicker lines
    \draw[thick, decorate, decoration={zigzag, segment length=7.5mm, amplitude=0.9375mm}] (A) -- (B);
    \draw[thick, decorate, decoration={zigzag, segment length=7.5mm, amplitude=0.9375mm}] (C) -- (D);
    
    % Draw diameter - scaled to 6cm with thicker line
    \draw[very thick] (-6,0) -- (6,0);
    
    % Mark points
    \node[left] at (-6,0) {$-1$};
    \node[right] at (6,0) {$1$};

    % 大弧线 - 它会被裁剪到指定区域
    \draw[very thick] (0,16.48) ++(250:17.54) arc (250:290:17.54);
  
    \draw[fill] (0,0) circle [radius=0.08];
    \node [above] at (0,0.05) {$v_0$};

    \draw[fill] (1.6,0) circle [radius=0.08];
    \node [above] at (1.4,0.05) {$v_1$};

    \draw[fill] (2.8,0) circle [radius=0.08];
    \node [above] at (2.8,0.05) {$v_2$};

    \draw[fill] (4.2,0) circle [radius=0.08];
    \node [above] at (4.2,0.05) {$v_3$};

    \draw[fill] (5.2,0) circle [radius=0.08];
    \node [above] at (5.2,0.05) {$v_4$};

    \draw[fill] (-1.2,0) circle [radius=0.08];
    \node [above] at (-1.2,0) {$v_{-1}$};

    \draw[fill] (-2.8,0) circle [radius=0.08];
    \node [above] at (-2.8,0) {$v_{-2}$};

    \draw [very thick]  (0.8,-1.03) circle [radius=1.13];
    \draw[fill]  (0.8,-1.03) circle [radius=0.08];

    \draw [very thick]  (3.5,-0.7) circle [radius=0.8];
    \draw[fill]  (3.5,-0.7) circle [radius=0.08];

    \draw [very thick]  (-2,-0.95) circle [radius=1.08];
    \draw[fill]  (-2,-0.95) circle [radius=0.08];

    \draw[very thick,dashed] (-2,-0.95)-- (-2,-0);
    \draw[very thick,dashed] (0.8,-1.03)-- (0.8,0);
    \draw[very thick,dashed] (3.5,-0.7)-- (3.5,0);
    
    \node [below] at (-2,-0.9) {$(R_{2}R_{1})^2(p_3)$};
    \node [below left] at (-2,0.1) {$p_3^{2}$};
    \node [below] at  (0.8,-1.03) {$R_{2}R_{1}(p_3)$};
    \node [below left] at  (0.8,0) {$p_3^{1}$};
    \node [below] at (3.5,-0.7) {$p_3$};
    \node [above] at (3.5,0) {$p_3^{0}$};
\end{tikzpicture}
\caption{Schematic diagram for \( p_3^0 = \Pi_{(-1, 1)}(p_3) \) in \( D_3 \).}
\label{figure 2k+1.png}
\end{figure}

\begin{prop}\label{simplify condition (3)}
    Suppose that \( m_3 > 0 \). If \( \Pi_{(-1, 1)}(p_3) \in D_l \), then condition~(3) in Proposition~\ref{proposition} is equivalent to the isometries \( w^{(l-1)} \) and \( w^{(l-2)} \) being non-elliptic.
\end{prop}

\begin{proof}
    For convenience, we define
    $$
    p_3^n = \Pi_{(-1, 1)}\left( (R_2 R_1)^n(p_3) \right).
    $$ 
    Suppose that \( \Pi_{(-1, 1)}(p_3) \in D_l \). Using equation~\eqref{Re(p_3)}, we know that \( \Re(p_3) > 0 \), which implies that \( l \) is a positive integer. We now consider two cases, depending on the parity of \( l \).

    \textbf{Case 1.} \( l = 2k + 1 \) for integers \( k \geq 0 \). Since \( \Pi_{(-1, 1)}(p_3) \in D_{2k+1} \) and the isometry \( R_2 R_1 \) restricted to \( C_{12} \) is loxodromic with axis \( (-1, 1) \) and translation length \( 2m_3 \), it follows that \( p_3^k \in D_1 \). 

    According to Lemmas~\ref{lem 4.4} and \ref{lem 4.5}, the condition that the isometries \( w^{(2k)} \) and \( w^{(2k-1)} \) are non-elliptic is equivalent to the requirement that \( (R_2 R_1)^k(S_3) \) does not contain the points \( v_1 \) and \( v_0 \), respectively. This, in turn, implies that \( (R_2 R_1)^n(S_3) \) does not contain any points \( v_j \) for all \( n \in \mathbb{Z} \). 

    By applying Lemma~\ref{lem 4.4}, we conclude that all isometries \( w^{(n)} \) are non-elliptic for \( n \in \mathbb{Z}^+ \). Conversely, if condition~(3) in Proposition~\ref{proposition} holds, then the isometries \( w^{(2k)} \) and \( w^{(2k-1)} \) are necessarily non-elliptic. See Figure~\ref{figure 2k+1.png}.

%\begin{figure}[htbp]
%    \centering
%    \includegraphics[scale=0.8]{figure_2k.png}
%    \caption{Schematic diagram for \( p_3^0 = \Pi_{(-1, 1)}(p_3) \) %in \( D_4 \).}
%    \label{figure 2k.png}
%\end{figure}

\begin{figure}[htbp]
   \centering
   \begin{tikzpicture}[scale=0.8]
    % Scale factor: 6/1.6 = 3.75
     \clip (-9,-5) rectangle (11,5);
    \def\scale{3.75}
    
    % Draw 130° to 230° arc (left part) - scaled to 6cm
    \draw[very thick] (140:6) arc[start angle=140, end angle=220, radius=6];
    
    % Draw -50° to 50° arc (right part) - scaled to 6cm
    \draw[very thick] (-40:6) arc[start angle=-40, end angle=40, radius=6];
    
    % Calculate coordinates of four endpoints - scaled to 6cm
    \coordinate (A) at (140:6);  % Upper left point
    \coordinate (B) at (40:6);   % Upper right point
    \coordinate (C) at (220:6);  % Lower left point
    \coordinate (D) at (-40:6);  % Lower right point
    
    % Draw zigzag lines - scaled proportionally with thicker lines
    \draw[thick, decorate, decoration={zigzag, segment length=7.5mm, amplitude=0.9375mm}] (A) -- (B);
    \draw[thick, decorate, decoration={zigzag, segment length=7.5mm, amplitude=0.9375mm}] (C) -- (D);
    
    % Draw diameter - scaled to 6cm with thicker line
    \draw[very thick] (-6,0) -- (6,0);
    
    % Mark points
    \node[left] at (-6,0) {$-1$};
    \node[right] at (6,0) {$1$};

\draw[very thick] (0,16.48) ++(250:17.54) arc (250:290:17.54);
  
\draw[fill] (0,0) circle [radius=0.08];
\node [above] at (0,0.05) {$v_0$};
\draw[fill] (1.3,0) circle [radius=0.08];
\node [above] at (1.3,0.05) {$v_1$};

\draw[fill] (2.9,0) circle [radius=0.08];
\node [above] at (2.9,0.05) {$v_2$};

\draw[fill] (4.1,0) circle [radius=0.08];
\node [above] at (4.1,0.05) {$v_3$};

\draw[fill] (5.3,0) circle [radius=0.08];
\node [above] at (5.3,0.05) {$v_4$};

\draw[fill] (-2,0) circle [radius=0.08];
\node [above] at (-2,0) {$v_{-1}$};

\draw[fill] (2.1,-0.95) circle [radius=0.08];
\draw [very thick] (2.1,-0.95) circle [radius=1.06];

\draw [very thick] (4.7,-0.4) circle [radius=0.5];
\draw[fill] (4.7,-0.4) circle [radius=0.08];

\draw [very thick]  (-1,-1.02) circle [radius=1.2];
\draw[fill]  (-1,-1.02) circle [radius=0.08];

\draw[very thick,dashed] (-1,-1.02)-- (-1,0);
\draw[very thick,dashed] (2.1,-0.95)-- (2.1,0);
\draw[very thick,dashed] (4.7,-0.4)-- (4.7,0);
\node [below] at (-1,-1.02) {$(R_{2}R_{1})^2(p_3)$};
\node [below left] at (-1,0) {$p_3^{2}$};
\node [below] at (2.1,-0.95) {$R_{2}R_{1}(p_3)$};
\node [below left] at (2.1,0) {$p_3^{1}$};
\node [below] at (4.7,-0.4) {$p_3$};
\node [above] at (4.7,0) {$p_3^{0}$};

\end{tikzpicture}
   \caption{Schematic diagram for \( p_3^0 = \Pi_{(-1, 1)}(p_3) \) in \( D_4 \).}
   \label{figure 2k.png}
\end{figure}

    \textbf{Case 2.} \( l = 2k \) for positive integers \( k \). Since the isometry \( R_2 R_1 \) restricted to \( C_{12} \) is loxodromic with axis \( (-1, 1) \) and translation length \( 2m_3 \), it follows that \( p_3^k \in D_0 \) and \( p_3^{k-1} \in D_2 \). 

    According to Lemmas~\ref{lem 4.4} and \ref{lem 4.5}, the condition that the isometry \( w^{(2k-1)} \) (resp. \( w^{(2k-2)} \)) is non-elliptic is equivalent to the requirement that \( (R_2 R_1)^k(S_3) \) (resp. \( (R_2 R_1)^{k-1}(S_3) \)) does not contain the point \( v_0 \) (resp. \( v_1 \)).

    We claim that \( v_{-1} \notin (R_2 R_1)^k(S_3) \). The reason is that if \( v_{-1} \in (R_2 R_1)^k(S_3) \), then we must have \( v_1 \in (R_2 R_1)^{k-1}(S_3) \), as \( R_1 R_2(v_{-1}) = v_1 \). This leads to a contradiction. 

    Since \( (R_2 R_1)^k(S_3) \) does not contain the points \( v_{-1} \) and \( v_0 \), it implies that \( (R_2 R_1)^n(S_3) \) does not contain any points \( v_j \) for all \( n \in \mathbb{Z} \). By applying Lemma~\ref{lem 4.4}, we conclude that all isometries \( w^{(n)} \) are non-elliptic for \( n \in \mathbb{Z}^+ \). Conversely, if condition~(3) in Proposition~\ref{proposition} holds, then the isometries \( w^{(2k-1)} \) and \( w^{(2k-2)} \) are necessarily non-elliptic. See Figure~\ref{figure 2k.png}.

\end{proof}

We then provide the coordinates of the points $v_j$ and $\Pi_{(-1, 1)}(p_3)$.

\begin{lem}
    The coordinates of the points \( v_j \) are given by
    \begin{equation}\label{vj}
        v_{j} = \frac{(r_3+s_3)^j - (r_3-s_3)^j}{(r_3+s_3)^j + (r_3-s_3)^j},\qquad
    \text{for}\,\,  j \in \mathbb{Z}.
    \end{equation}
\end{lem}
\begin{proof}
    This follows from a direct computation, noting that  
\[
    v_{2k} = (R_1 R_2)^k(v_0) \qquad \text{and} \qquad  
    v_{2k+1} = (R_1 R_2)^k(v_1),
\]
where the expression for \( (R_2 R_1)^k \), and hence for \( (R_1 R_2)^k \), is given in Lemma~\ref{translation length}.
\end{proof}
\medskip

\begin{lem}\label{proj of p3}
We have \( \Pi_{(-1, 1)}(p_3) = f(\cos \alpha) \), where
\begin{equation}\label{p3'}
    f(t) = \frac{\sqrt{x_0 - x_1 t} - \sqrt{y_0 - y_1 t}}{\sqrt{x_0 - x_1 t} + \sqrt{y_0 - y_1 t}},
\end{equation}
and
\begin{equation*}
\begin{aligned}
x_0 &= r_1^2 (r_3 + s_3)^2 + r_2^2, \quad & x_1 &= 2 r_1 r_2 (r_3 + s_3), \\
y_0 &= r_1^2 (r_3 - s_3)^2 + r_2^2, \quad & y_1 &= 2 r_1 r_2 (r_3 - s_3).
\end{aligned}
\end{equation*}
Moreover, \( f(t) \) is an increasing function of \( t \), and thus \( \Pi_{(-1, 1)}(p_3) \in [f(-1), f(1)] \).
\end{lem}

\begin{proof}
 We begin by computing the value of \( \Pi_{(-1, 1)}(p_3) \).
 Let \( \mathbf{H}_{\mathbb{R}}^2 \) denote the upper half-plane. The isometry \( F(z) = \frac{z - i}{z + i} \) maps \( \mathbf{H}_{\mathbb{R}}^2 \) to \( \mathbf{H}_{\mathbb{C}}^1 \). A simple computation yields the inverse of \( F \):
 
    \[
    F^{-1}(z) =\frac{z+1}{i(z-1)}.
    \]
    Consequently,
\[
|F^{-1}(z)|
=\frac{|z+1|}{|z-1|}.
\]
Therefore, the orthogonal projection of \( F^{-1}(z) \) onto the geodesic line \( i\mathbb{R}^+ \) is given by \( |F^{-1}(z)| i \), and thus the projection \( \Pi_{(-1, 1)}(z) \) is
    \[
    \Pi_{(-1, 1)}(z) = F(|F^{-1}(z)| i).
    \]
    Explicitly,
    \[
    \Pi_{(-1, 1)}(z) =\frac{|z+1|-|z-1|}{|z+1|+|z-1|}.
    \]
Hence, the value of $\Pi_{(-1,1)}(p_3)$ can be obtained by directly substituting $p_3$ from equation~\eqref{p3 point 2}.

We then consider the function \( f(t) \) defined in \eqref{p3'}.  
Differentiating with respect to \( t \), we obtain
\[
    f'(t) = \frac{x_0 y_1 - x_1 y_0}{\left( \sqrt{x_0 - x_1 t} + \sqrt{y_0 - y_1 t} \right)^2 \sqrt{x_0 - x_1 t}\,\sqrt{y_0 - y_1 t}}.
\]
Clearly, the sign of \( f'(t) \) depends solely on the numerator. Substituting the explicit values of \( x_j \) and \( y_j \) for \( j = 0, 1 \), we get
\[
    x_0 y_1 - x_1 y_0 = 4 r_1 r_2 s_3 (r_1^2 - r_2^2) \geq 0.
\]
Hence, \( f(t) \) is an increasing function of \( t \), and therefore \( \Pi_{(-1, 1)}(p_3) \in [f(-1), f(1)] \).
\end{proof}

Finally, we arrive at the number $l$ in Proposition~\ref{simplify condition (3)}.

\begin{lem}\label{number l}
    We have \( \Pi_{(-1, 1)}(p_3) \in D_l=(v_{l-1}, v_l] \), where
    \[
    l=\mathrm{ceil} \left(\frac{\log\left(1 + \Pi_{(-1, 1)}(p_3)\right) - \log\left(1 - \Pi_{(-1, 1)}(p_3)\right)}{\log\left(r_3 + s_3\right) - \log\left(r_3 - s_3\right)}\right),
    \]
with $\mathrm{ceil}\,(\cdot)$ denoting the ceiling function, and the value of $\Pi_{(-1,1)}(p_3)$ is given by~\eqref{p3'}.
\end{lem}

\begin{proof}
    The inequality  
\[
    v_k = \frac{(r_3 + s_3)^k - (r_3 - s_3)^k}{(r_3 + s_3)^k + (r_3 - s_3)^k} < \Pi_{(-1, 1)}(p_3),
\]
is equivalent to  
\[
    k < \frac{\log\!\big(1 + \Pi_{(-1, 1)}(p_3)\big) - \log\!\big(1 - \Pi_{(-1, 1)}(p_3)\big)}{\log(r_3 + s_3) - \log(r_3 - s_3)}.
\]
This immediately yields the result.
\end{proof}
\medskip

\section{A new family of complex hyperbolic $[m_1, m_2, m_3]$-triangle group representations with $m_3 > 0$}\label{sec m3>0}
In this section, we prove Theorem~\ref{main thm 2} by applying Proposition~\ref{proposition} together with the algorithm presented in Section~\ref{algorithm}. 

We provide the trace of $w^{(n)}=R_1(R_2R_1)^nR_3$ by using the following lemma.

\begin{lem}\label{translation length}
    If $m_3 > 0$, then the isometry $R_2R_1$, when restricted to the complex geodesic $C_{12}$, is hyperbolic and has a translation length of $2m_3$, where $m_3$ is the distance between the complex geodesics $C_1$ and $C_2$. Additionally, we have
    \begin{equation}\label{(R2R1)n}
        (R_2R_1)^n=
        \begin{bmatrix}
 1 & 0 & 0\\
 0 & \frac{(r_3-s_3)^{2n}+(r_3+s_3)^{2n}}{2} & \frac{(r_3-s_3)^{2n}-(r_3+s_3)^{2n}}{2}\\
 0 & \frac{(r_3-s_3)^{2n}-(r_3+s_3)^{2n}}{2} & \frac{(r_3-s_3)^{2n}+(r_3+s_3)^{2n}}{2}
\end{bmatrix}.
    \end{equation}
\end{lem}

\begin{proof}
    Since the isometry $R_2R_1$, when restricted to $C_{12}$, is the product of two involutions about the points $p_2$ and $p_1$, where the points $p_j$ for $j=1,2$ are the orthogonal projections of $C_j$ onto $C_{12}$, the isometry $R_2R_1$ is hyperbolic, with the geodesic line passing through $p_1$ and $p_2$ as its axis and with translation length equal to $2m_3$, where $m_3$ denotes the hyperbolic distance between $p_1$ (resp. $C_1$) and $p_2$ (resp. $C_2$); see the last paragraph of Section~7.34 in \cite{bea}.

    We then compute the matrix expression of $(R_2R_1)^n$. Observe that
    \[
    R_2R_1 =
    \begin{bmatrix}
    1 & 0 & 0 \\
    0 & r_3^2 + s_3^2 & -2r_3 s_3 \\
    0 & -2r_3 s_3 & r_3^2 + s_3^2
    \end{bmatrix}.
    \]
    When restricted to the complex geodesic \( C_{12} \), the isometry \( R_2R_1 \) acts on \( C_{12} \) as
    \[
    R_{21}' =
    \begin{bmatrix}
    r_3^2 + s_3^2 & -2r_3 s_3 \\
    -2r_3 s_3 & r_3^2 + s_3^2
    \end{bmatrix} \in \mathrm{SU}(1,1).
    \]
    Note that
    \[
    A^{-1} (R_{21}')^n A = (A^{-1} R_{21}' A)^n =
    \begin{bmatrix}
    (r_3 - s_3)^{2n} & 0 \\
    0 & (r_3 + s_3)^{2n}
    \end{bmatrix},
    \]
    where
    \[
    A =
    \begin{bmatrix}
    1 & -i \\
    1 & i
    \end{bmatrix}.
    \]
    Hence, we get
    \[
    (R_{21}')^n =
    \begin{bmatrix}
    \frac{(r_3 - s_3)^{2n} + (r_3 + s_3)^{2n}}{2} & \frac{(r_3 - s_3)^{2n} - (r_3 + s_3)^{2n}}{2} \\
    \frac{(r_3 - s_3)^{2n} - (r_3 + s_3)^{2n}}{2} & \frac{(r_3 - s_3)^{2n} + (r_3 + s_3)^{2n}}{2}
    \end{bmatrix},
    \]
    which yields \eqref{(R2R1)n} directly.
\end{proof}
\medskip

A straightforward computation yields
\begin{equation*}
    \mathrm{tr}(w^{(n)})=\frac{A_0'+A_1'+A_2'}{s_3^3},
\end{equation*}
where
\begin{equation*}
\begin{aligned}
A_0'&=(r_3 - s_3)^{2n + 2}\left((2r_1^2r_3^2 - r_1^2 + r_2^2)s_3 + 2r_1^2r_3(r_3^2 - 1)-2r_1r_2s_3(r_3 + s_3)\cos\alpha\right),\\
A_1'&=(r_3 + s_3)^{2n + 2}\left((2r_1^2r_3^2 - r_1^2 + r_2^2)s_3 - 2r_1^2r_3(r_3^2 - 1)-2r_1r_2s_3(r_3 - s_3)\cos\alpha\right),\\
A_2'&=s_3\left(4r_1r_2r_3\cos\alpha - 2r_1^2 - 2r_2^2 - r_3^2 + 1 \right).
\end{aligned}
\end{equation*}
Then, replacing $\cos\alpha$ by $t_n$, and then solving the equation $\mathrm{tr}(w^{(n)})=3$, we obtain 
\begin{equation}\label{tn}
    t_n=\frac{A_0+A_1+A_2}{2r_1r_2s_3\left((r_3+s_3)^{2n+1}+(r_3-s_3)^{2n+1}-2r_3\right)},
\end{equation}
where \( A_0 \), \( A_1 \), and \( A_2 \) are given by
\begin{equation*}\label{A0A1A2}
\begin{aligned}
A_0&=(r_3 - s_3)^{2n + 2}\left((2r_1^2r_3^2 - r_1^2 + r_2^2)s_3 + 2r_1^2r_3(r_3^2 - 1)\right),\\
A_1&=(r_3 + s_3)^{2n + 2}\left((2r_1^2r_3^2 - r_1^2 + r_2^2)s_3 - 2r_1^2r_3(r_3^2 - 1)\right),\\
A_2&=-2s_3(r_1^2 + r_2^2 + 2r_3^2 - 2).
\end{aligned}
\end{equation*}
Notice that if $t_n\le 1$, the isometry $w^{(n)}$ is unipotent parabolic when the angular invariant $\alpha=\arccos(t_n)$; otherwise, the isometry $w^{(n)}$ is always loxodromic for all values of the angular invariant $\alpha$.

We define two kinds of sets.
\begin{defn}\label{defn of Kn}
Let \( r_3>1 \) and a positive integer \( k_0 \) be fixed. 
We define the sets  
\begin{equation}\label{K_n'}
    K_n' = \left\{ (r_1, r_2, \alpha) :\alpha\in[0, \pi],\,1 < r_2 \leq r_1 \,\text{ and }\, t_n \leq t_k \text{ for every } k \neq n \right\},
\end{equation}
where $\alpha$ is the angular invariant, \( n \) and \( k \) are positive integers satisfying \( 1 \leq n \leq k_0 \) and \( 1 \leq k \leq k_0 + 1 \).
The sets \( K_n \) are then defined as the subsets of \( K_n' \) satisfying the following conditions:
\begin{align}
r_1^2 + r_2^2 - 2r_1r_2r_3t_n &\geq 0, \label{eq:first} \\
2r_3^2(r_1^2 + r_2^2) - 4r_1r_2r_3t_n - 1 &\geq 0, \label{eq:second} \\
a_2 t_n^2 + a_1 t_n + a_0 &\geq 0, \label{eq:third} \\
v_{k_0+2} - f(1) &\geq 0. \label{eq:fourth}
\end{align}
Here, \( a_j \), \( v_n \), and \( t_n \) are given in equations~\eqref{eq:a0}, \eqref{eq:a1a2}, \eqref{vj}, and \eqref{tn}, respectively, and the function \( f(t) \) is defined in~\eqref{p3'}.
\end{defn}

We remark that the inequality~\eqref{eq:second} can be omitted if $r_3 \ge \sqrt{\frac{1+\sqrt{2}}{2}}$. The reason is explained below. We observe that the inequalities~\eqref{eq:first} and \eqref{eq:second} are equivalent to
\[
t_n \le \frac{r_1^2 + r_2^2}{2 r_1 r_2 r_3}
\qquad \text{and} \qquad
t_n \le \frac{2 r_3^2 (r_1^2 + r_2^2) - 1}{4 r_1 r_2 r_3},
\]
respectively. Moreover, we have
\begin{equation*}
\begin{aligned}
\frac{r_1^2 + r_2^2}{2 r_1 r_2 r_3}
\le \frac{2 r_3^2 (r_1^2 + r_2^2) - 1}{4 r_1 r_2 r_3}
&\Longleftrightarrow 2 (r_1^2 + r_2^2) (r_3^2 - 1) \ge 1 \\
&\Longleftarrow 4 r_3^2 (r_3^2 - 1) \ge 1 \\
&\Longleftrightarrow r_3 \ge \sqrt{\frac{1+\sqrt{2}}{2}} .
\end{aligned}
\end{equation*}
Therefore, when $r_3 \ge \sqrt{\frac{1+\sqrt{2}}{2}}$, inequality~\eqref{eq:first} implies inequality~\eqref{eq:second}.

We now can give the proof of Theorem~\ref{main thm 2}.
\medskip

\begin{proof}[Proof of Theorem~\ref{main thm 2}]
The ``only if" direction is straightforward. Thus, it suffices to prove the ``if" direction.

Note that the isometry $w^{(n)}=R_1 (R_2 R_1)^n R_3$ is non-elliptic if and only if \(\cos \alpha \le t_n\).
Moreover, if  
\[
t_n \le \frac{r_1^2 + r_2^2}{2 r_1 r_2 r_3},
\]
which is equivalently written as  
\begin{equation*}
    r_1^2 + r_2^2 - 2 r_1 r_2 r_3 t_n \ge 0,
\end{equation*}
then the non-ellipticity of $w^{(n)}$ implies condition~(1) in Proposition~\ref{proposition}.

    We then consider the quadratic function $g(t) = a_2 t^2 + a_1 t + a_0$, where the coefficients $a_j$ for $j = 0, 1, 2$ are defined in equations~\eqref{eq:a0} and~\eqref{eq:a1a2}, and it is clear that $a_2 > 0$.
 The axis of the quadratic function $g(t)$ is the vertical line 
    $$
    t=-\frac{a_1}{2a_2}=\frac{2r_3^2(r_1^2+r_2^2)-1}{4r_1r_2r_3}.
    $$
    The following inequality
    $$
    t_n\le \frac{2r_3^2(r_1^2+r_2^2)-1}{4r_1r_2r_3} 
    $$
    is equivalent to
    \begin{equation}\label{c 2}
        2r_3^2(r_1^2+r_2^2)-4r_1r_2r_3t_n-1\ge0.
    \end{equation}
    Since $t=-\frac{a_1}{2a_2}$ is the axis of symmetry of $g(t)$, we have that if the inequality~\eqref{c 2} holds, then  
\begin{equation*}
    g(\cos\alpha) \geq g(t_n),
\end{equation*}
for $\cos\alpha \leq t_n$.  
Moreover, if we further impose the condition  
\begin{equation*}
    g(t_n) = a_2 t_n^2 + a_1 t_n + a_0 \geq 0,
\end{equation*}
then condition~$(2)$ of Proposition~\ref{proposition} follows.

Finally, we consider inequality~\eqref{eq:fourth}. According to Lemma~\ref{proj of p3}, we have 
$$
\Pi_{(-1, 1)}(p_3) = f(\cos\alpha),
$$
where the function $f(t)$, given in \eqref{p3'}, is increasing. Note that $\Pi_{(-1, 1)}(p_3)$ ranges from $f(-1)$ to, at most, $f(1)$.
If $f(1) \leq v_{k_0+2}$, then, according to Proposition~\ref{simplify condition (3)} and Lemma~\ref{proj of p3}, in order for condition~$(3)$ of Proposition~\ref{proposition} to hold, it suffices to consider at most the isometric types of $w^{(1)}$, $w^{(2)}$, $\dots$, $w^{(k_0)}$ and $w^{(k_0+1)}$. It follows from the construction of $K_n'$ given in \eqref{K_n'} that $w^{(n)}$ becomes elliptic first. Hence, if $f(1) \leq v_{k_0+2}$ and $w^{(n)}$ is non-elliptic, then condition~$(3)$ of Proposition~\ref{proposition} is satisfied.

Therefore, in each $K_n'$, the corresponding complex hyperbolic $[m_1, m_2, m_3]$-triangle group representations, satisfying inequalities~\eqref{eq:first}, \eqref{eq:second}, \eqref{eq:third} and~\eqref{eq:fourth}, are discrete and faithful if and only if $w^{(n)}$ are non-elliptic.
\end{proof}

We set \( r_3 = 1.01 \) (resp. $r_3=1.09$) and \( k_0 = 3 \), and then visualize the sets \( K_n \) for \( n = 1, 2, 3 \), as shown in Figure~\ref{ultra-picture1.png} (resp. Figure~\ref{ultra-picture2.png}). Since the sets $K_n$ are defined by several inequalities, these figures can be generated computationally.
Each set \( K_n \) is divided into two regions: the first corresponds to representations where \( t_n > 1 \), while the second corresponds to those where \( t_n \leq 1 \). It is important to note that the condition \( t_n > 1 \) ensures that the isometry \( w^{(n)} \) is always loxodromic for all values of the angular invariant \( \alpha \). Thus, the representations in the first region (that is, the black region in Figures~\ref{ultra-picture1.png} and~\ref{ultra-picture2.png}) are always discrete and faithful.

\medskip

\section{Discreteness of the complex hyperbolic $[m_1, m_2, 0]$-triangle group representations}\label{sec m3=0}
In this section, we study the complex hyperbolic $[m_1, m_2, 0]$-triangle group representations and prove Theorem~\ref{main thm 3}.
We use the same parameterization as in \cite{mpp}.  After normalization, the normalized polar vectors $n_j$ of the mirrors of the generators $R_j$ can be written as
\begin{equation*}
n_1=
\begin{bmatrix}
\sqrt{2}r_2e^{-i\theta}\\
1 \\
0
\end{bmatrix},\quad
n_2=
\begin{bmatrix}
-\sqrt{2}r_1e^{i\theta}\\
1 \\
0
\end{bmatrix},\quad
n_3=
\begin{bmatrix}
1/\sqrt{2}\\
0 \\
1/\sqrt{2}
\end{bmatrix}.
\end{equation*}
According to Remark \ref{remk angular}, we can restrict the angular invariant $\alpha$ to the interval $[0, \pi]$. Since $\theta = \frac{\pi - \alpha}{2}$ (see the initial paragraph of Section~3 in \cite{mpp}), we have $\theta \in \left[ 0, \frac{\pi}{2} \right]$. We can also express the matrices as:
\begin{equation*}
R_1=
\begin{bmatrix}
 -1 & 2\sqrt{2}\,r_2e^{-i\theta} & 4r_2^2\\
 0 & 1 & 2\sqrt{2}\,r_2e^{i\theta}\\
 0 & 0 & -1
\end{bmatrix},\quad
R_2=
\begin{bmatrix}
 -1 & -2\sqrt{2}\,r_1e^{i\theta} & 4r_1^2\\
 0 & 1 & -2\sqrt{2}\,r_1e^{-i\theta}\\
 0 & 0 & -1
\end{bmatrix},
\end{equation*}
\begin{equation*}
R_3=
\begin{bmatrix}
0  & 0 & 1 \\
0  & -1 & 0\\
1  & 0 & 0
\end{bmatrix},\quad
R_2R_1=
\begin{bmatrix}
 1 & -2\sqrt{2}(r_1e^{i\theta}+r_2e^{-i\theta}) & -4(r_1^2+r_2^2)-8r_1r_2e^{2i\theta}\\
 0 & 1 & 2\sqrt{2}(r_1e^{-i\theta}+r_2e^{i\theta})\\
 0 & 0 & 1
\end{bmatrix},
\end{equation*}
where $R_2R_1$ is the Heisenberg translation by $[2(r_1e^{-i\theta}+r_2e^{i\theta}), -8r_1r_2\sin(2\theta)]$.

In the boundary \( \partial \mathbf{H}_{\mathbb{C}}^2 \), which can be identified with \( \mathcal{H} \cup \{ \infty \} \), the polar vectors \( n_1 \) and \( n_2 \) correspond to the infinite chains
\[
C_1: \zeta = -r_2 e^{i \theta}, \quad C_2: \zeta = r_1 e^{-i \theta},
\]
respectively. The polar vector \( n_3 \) corresponds to the finite chain, which is the unit circle in the plane \( \mathbb{C} \times \{ 0 \} \).

The Cygan sphere of radius 1 centered at \( o = [0, 0] \in \mathcal{H} \), denoted by \( S \), divides \( \partial \mathbf{H}_{\mathbb{C}}^2 \) into two components, either of which can serve as a fundamental domain in \( \partial \mathbf{H}_{\mathbb{C}}^2 \) for the group \( \langle R_3 \rangle \). We denote by \( S_1 \) the vertical projection of \( S \), i.e., the unit disc in the plane \( \mathbb{C} \times \{ 0 \} \), and by \( S^\circ \), the interior of \( S \). Based on a similar argument provided in Section~4.1, the following lemma is obtained.

\begin{lem}
    A complex hyperbolic \( [m_1, m_2, 0] \)-triangle group is discrete and is the free product of the groups \( \langle R_j \rangle \) for \( j = 1, 2, 3 \), if the following conditions are satisfied:
    \begin{enumerate}
        \item[(1)] \( S^\circ \) does not intersect \( (R_2 R_1)^n (S^\circ) \) for all \( n \in \mathbb{Z} \setminus \{ 0 \} \);

        \item[(2)] \( S^\circ \) does not intersect \( R_1 (R_2 R_1)^n (S^\circ) \) for all \( n \in \mathbb{Z} \).
    \end{enumerate}
\end{lem}
\medskip

The images of \( C_1 \) and \( C_2 \) under the vertical projection are the points
$$
\mathrm{p}_1 = -r_2 e^{i\theta} \quad \text{and} \quad
\mathrm{p}_2 = r_1 e^{-i\theta},
$$
respectively. Let \( \mathcal{L} \subset \mathbb{C} \times \{0\} \) be the straight line passing through \( \mathrm{p}_1 \) and \( \mathrm{p}_2 \). We denote the points on this line by
$$
\mathrm{v}_k = \mathrm{p}_2 + k(\mathrm{p}_1 - \mathrm{p}_2) = r_1 e^{-i\theta} - k(r_2 e^{i\theta} + r_1 e^{-i\theta}),
$$
where \( k \) is an integer. Under this notation, we have \( \mathrm{v}_0 = \mathrm{p}_2 \) and \( \mathrm{v}_1 = \mathrm{p}_1 \). For convenience, we denote the infinite chains whose vertical projections are \( \mathrm{v}_k \) by \( \mathcal{C}_k \). It is clear that
$$
R_2 R_1 (\mathcal{C}_k) = \mathcal{C}_{k-2}.
$$
Since \( r_1 \ge r_2 \), it follows that in the plane \( \mathbb{C} \times \{0\} \), the Euclidean projection of the center of \( S_1 \) (i.e., the origin \( o \)) onto the straight line \( \mathcal{L} \) must lie in the segment \( [\mathrm{v}_k, \mathrm{v}_{k+1}] \) for some integer \( k \ge 0 \). See Figure~\ref{figure m3.png}.

At this point, we are back to a similar case as discussed in Section 4. Here, the straight line \( \mathcal{L} \) and the points \( \mathrm{v}_k \) play the same role as the axis \( (-1, 1) \) and the points \( v_k \) defined in Section 4, respectively. The proof of Lemma~\ref{[m1, m2, 0] lem} is not essentially different from the proofs of Lemmas~\ref{lem 4.4} and \ref{lem 4.5}, except that the bisector is replaced by a Cygan sphere. Therefore, we omit the details of the proof.

%\begin{figure}[htbp]
%    \centering
%     \includegraphics[scale=0.7]{figure_m3.png}
%     \caption{}
%     \label{figure m3.png}
% \end{figure}

\begin{figure}[htbp]
    \centering
    \begin{tikzpicture}[scale=0.8]
        \draw [very thick,->] (-4.5,0)-- (4.5,0);
        \draw [very thick,->] (0,-4.5)-- (0,4.5);
        \node [left] at (0,4.5) {$y$};
        \node [below right] at (4.5,0) {$x$};
        \node [above right] at (0,0) {$o$};
        \draw [very thick]  (0,0) circle [radius=1.8];
      \node [above] at (1.42,1.43) {$S_1$};
    \draw[very thick,dashed] (-3,3.5)-- (0.5,-5.25);
     \draw[very thick,dashed] (-1.38,-0.55)-- (0,0);  
\node [left] at (-3,3.7) {$\mathcal{L}$};
 \draw[fill] (-2.7,2.75) circle [radius=0.08];
\draw[fill] (-1.9,0.75) circle [radius=0.08];
\draw[fill] (-0.9,-1.75) circle [radius=0.08];
\draw[fill] (-0.2,-3.5) circle [radius=0.08];
 \draw[fill] (0.3,-4.75) circle [radius=0.08];
\node [left] at (-3,3.7) {$\mathcal{L}$};
\node [left] at (-2.7,2.75) {$v_4$};
 \node [left] at (-1.9,0.75) {$v_3$};
 \node [left] at (-0.9,-1.75) {$v_2$};
 \node [left] at (-0.2,-3.5) {$v_1=p_1=-r_2e^{i\theta}$};
  \node [right] at (0.3,-4.75) {$v_0=p_2=-r_1e^{-i\theta}$};
    \end{tikzpicture}
    \caption{}
    \label{figure m3.png}
\end{figure}

\begin{lem}\label{[m1, m2, 0] lem}
    The complex hyperbolic $[m_1, m_2, 0]$-triangle group representations are discrete and faithful if the following conditions are satisfied:
    \begin{enumerate}
        \item[(1)] $S^\circ$ does not intersect $(R_2R_1)^n(S^\circ)$ for all $n\in \mathbb{Z^+}$;

        \item[(2)] the isometries $w^{(n)}=R_1(R_2R_1)^nR_3$ are non-elliptic for all $n\in \mathbb{Z^+}$.
    \end{enumerate}
\end{lem}

Subsequently, we provide more precise descriptions of condition~$(1)$ and $(2)$ in Lemma~\ref{[m1, m2, 0] lem}.

\begin{lem}\label{Cygan distant}
    Condition~(1) in Lemma~\ref{[m1, m2, 0] lem} is satisfied if \( h(\cos \alpha) \ge 0 \), where
    \begin{equation*}
    h(t) = -4r_1r_2(r_1^2 + r_2^2)t + r_1^4 + 6r_1^2r_2^2 + r_2^4 - 1.
    \end{equation*}
\end{lem}

\begin{proof}
    The isometry \( R_2 R_1 \) acts on \( \partial \mathbf{H}_{\mathbb C}^2 \) as a Heisenberg translation. Hence, the condition that \( S^\circ \) does not intersect \( (R_2 R_1)^n(S^\circ) \) for all \( n \in \mathbb{Z}^+ \) is equivalent to the condition that \( S^\circ \) does not intersect \( R_2 R_1(S^\circ) \). A direct computation shows that the center of the Cygan sphere \( R_2 R_1(S) \) is given by
    %$$
    %R_2 R_1(o) =
    %\begin{bmatrix}
    % -4(r_1^2 + r_2^2) - 8r_1r_2 e^{2i\theta} \\
    %  2\sqrt{2}(r_1 e^{-i\theta} + r_2 e^{i\theta}) \\
    %  1
    % \end{bmatrix} \sim
    % [2(r_1 e^{-i\theta} + r_2 e^{i\theta}), -8r_1r_2 \sin(2\theta)] %\in \mathcal{H}.
    % $$
   $$
    R_2 R_1(o) =
    [2(r_1 e^{-i\theta} + r_2 e^{i\theta}), -8r_1r_2 \sin(2\theta)] \in \mathcal{H}.
    $$ 
    Therefore, \( S^\circ \) does not intersect \( R_2 R_1(S^\circ) \) if
    $$
    \rho_0^4(o, R_2 R_1(o)) - (1 + 1)^4 \ge 0,
    $$
    which simplifies to
    $$
    8r_1r_2(r_1^2 + r_2^2) \cos^2 \theta + r_1^4 - 4r_1^3 r_2 + 6r_1^2 r_2^2 - 4r_1 r_2^3 + r_2^4 - 1 \ge 0.
    $$
    Substituting \( \theta = \frac{\pi - \alpha}{2} \), we obtain the equivalent condition
    $$
    -4r_1 r_2 (r_1^2 + r_2^2) \cos \alpha + r_1^4 + 6r_1^2 r_2^2 + r_2^4 - 1 \ge 0.
    $$
\end{proof}

According to \cite{mpp}, if \( m_2 > 0 \), then the pair \( (r_1, r_2) \) corresponds to the following:
\begin{equation*}
    (X, Y) = \left( \frac{r_1^2 - 1}{r_2^2 - 1} - 1, \frac{1}{r_2^2 - 1} \right).
\end{equation*}
Equivalently, we have
\begin{equation}\label{(r1, r2)}
    (r_1, r_2) = \left( \sqrt{1 + \frac{X + 1}{Y}}, \sqrt{1 + \frac{1}{Y}} \right).
\end{equation}
We define the following sets.
\begin{defn}\label{defn of KN}
The set $\mathcal{K}_n'$ is defined by
\begin{equation}\label{KKn}
\mathcal{K}_n' = \left\{ (X, Y, \alpha) \mid \frac{2}{n} \le X \le \frac{2}{n - 1},\, Y > 0, \alpha\in[0, \pi] \right\},
\end{equation}
for \( n \in \mathbb{Z}^+ \), where the inequality \( X \le \frac{2}{n - 1} \) is omitted in the case \( n = 1 \).
The set $\mathcal{K}_n$ is then defined as the subset of $\mathcal{K}_n'$ satisfying the additional condition
\begin{equation}\label{Kn condition}
- n^2 (X^2 + Y^2) + n \bigl( X^2 - Y^2 + 4 (Y + 1) X \bigr) - 2 X - 4 Y - 4 \ge 0.
\end{equation}
\end{defn}

\medskip

\begin{lem}\label{wn before}
    If 
    $$
    X=\frac{r_1^2-1}{r_2^2-1}-1\in \left[\frac{2}{n}, \frac{2}{n-1}\right],
    $$
    where $r_2>1$,
    then the isometry $w^{(n)}$ becomes elliptic before $w^{(m)}$ for any $m\in \mathbb{Z^+}\setminus \left\{n \right\}$.
\end{lem}
\begin{proof}
    According to Lemma~$8$ of \cite{mpp}, we have
    $$
    \mathrm{tr}(w^{(n)})=4(nr_1-(n+1)r_2)^2-1+8n(n+1)r_1r_2(1-\cos\alpha).
    $$
    A simple computation shows that when 
    \begin{equation}\label{tn m3=0}
    \cos\alpha=t_n=\frac{(r_1^2+r_2^2)n^2+r_2^2(2n+1)-1}{2r_1r_2n(n+1)},
    \end{equation}
    the isometry $w^{(n)}$ is unipotent parabolic, equivalently, 
    $$
    \mathrm{tr}(w^{(n)})=3.
    $$
    The isometry $w^{(n)}$ becomes elliptic before $w^{(m)}$ where $n\ne m$ if and only if $t_n\le t_m$. We compute that
    \begin{equation}\label{H(n)}
        t_{n-1}-t_{n}=\frac{n(r_2^2-r_1^2)+r_1^2+r_2^2-2}{2r_1r_2n(n^2-1)},
    \end{equation}
    for $n\ge 2$.
    Substituting 
    $$
    r_1=\sqrt{1+\frac{X+1}{Y}}\qquad \text{and} \qquad
    r_2=\sqrt{1+\frac{1}{Y}},
    $$
    where $X\ge0$ and $Y>0$,
    into the right hand side of the equality~\eqref{H(n)}, we obtain
    $$
    t_{n-1}-t_{n}=\frac{X(1-n)+2}{2n(n^2-1)\sqrt{(1+Y)(1+X+Y)}}.
    $$
    Hence, the isometry $w^{(n)}$ becomes elliptic before $w^{(n-1)}$ if and only if
    \begin{equation}\label{X}
        X\le \frac{2}{n-1},
    \end{equation}
    where $n\in \mathbb{Z^+}$ and $n\ge 2$. 
    
    We claim that if 
    $$
    \frac{2}{n}\le X\le \frac{2}{n-1},
    $$
    then the isometry $w^{(n)}$ becomes elliptic before $w^{(m)}$ for any positive integers $m\ne n$. If $n=1$, the inequality $X\le \frac{2}{n-1}$ should be omitted. Suppose that $m_0$ is any positive integer satisfying $m_0 > n$. According to \eqref{X}, if $X\geq\frac{2}{n}$,  
then $t_n\leq t_{n+1}$. Applying inequality~\eqref{X} again, we obtain that if  
$$
X\geq \frac{2}{n}\geq\frac{2}{n+1},
$$  
then $t_n\leq t_{n+1}\leq t_{n+2}$. Repeating this process iteratively, after a finite number of steps, we have  
$$
t_n\leq t_{n+1}\leq t_{n+2}\leq \dots \leq t_{m_0-1}\leq t_{m_0},
$$  
if  
$$
X\geq \frac{2}{n}\geq\frac{2}{n+1}\geq \dots \geq \frac{2}{m_0-1}.
$$  
Similarly, we deduce that $t_n\leq t_{n-1}\leq \dots \leq t_1$ if $X\leq\frac{2}{n-1}$. Thus, the result follows.
\end{proof}

According to Lemma~\ref{wn before}, within each region \( \mathcal{K}_n' \), condition~(2) in Lemma~\ref{[m1, m2, 0] lem} is equivalent to the isometry \( w^{(n)} \) being non-elliptic, that is, \( \cos \alpha \leq t_n \). It is straightforward to observe that the function \( h(t) \) defined in Lemma~\ref{Cygan distant} decreases as \( t \) increases. Therefore, we get
\[
h(t) \geq h(t_n) = \frac{-n^2 \left( (r_1^2 - r_2^2)^2 + 1 \right) + n \left( (r_1^2 + r_2^2)^2 - 4 r_2^4 - 1 \right) - 2 (r_1^2 + r_2^2)(r_2^2 - 1)}{n(n+1)},
\]
for \( t = \cos \alpha \leq t_n \). This implies that within each \( \mathcal{K}_n' \), by Lemma~\ref{Cygan distant}, the condition \( h(t_n) \geq 0 \) ensures condition~(1) in Lemma~\ref{[m1, m2, 0] lem}. 
Moreover, using \eqref{(r1, r2)}, the inequality \( h(t_n) \ge 0 \) is equivalent to  inequality~\eqref{Kn condition} in Definition~\ref{defn of KN}. 
Thus, combining this result with the discussion above, we conclude the proof of Theorem~\ref{main thm 3}

\medskip

\section{The limit of a subset of $K_n$ as $r_3$ tends to $1$}\label{sec 7}
We begin by defining a subset of $K_n$ for $n \ge 2$.

\begin{defn}\label{Kn''}
For $n \ge 2$, using the $(X, Y)$-coordinates, the set $\tilde{K}_n$ is defined as the subset of $K_n$ (see Definition~\ref{defn of Kn}) that satisfies the additional conditions
$$
\frac{2}{n} \le X \le \frac{2}{n-1}
\quad \text{and} \quad
nX - n(n+1)Y - 1 \ge 0.
$$
\end{defn}

Our main result in this section is the following:
\begin{prop}
Suppose that $\mathcal{P}_n$, defined in equation~\eqref{Pn}, is the set obtained in Proposition~1 of \cite{mpp}, where $n\in \mathbb{Z}^+$. Then the following statements hold:
\begin{enumerate}
\medskip
\item[(a)] For any given positive integer $n \ge 2$, the set $\mathcal{P}_n$ is the limit of $\tilde{K}_n$ as $r_3$ tends to $1$;
\medskip
\item[(b)] Any compact subset of $\mathcal{P}_1$ (in the $(X,Y)$-coordinates) is the limit of a compact subset of $K_1$ as $r_3$ tends to $1$.
\end{enumerate}
\end{prop}

\begin{proof}
 We first consider~$\mathrm{(a)}$. By equation~\eqref{tn}, the quantity $t_n$ tends to the value given in~\eqref{tn m3=0} as $r_3$ tends to $1$. Consequently, by Lemma~\ref{wn before}, the set $\mathcal{K}_n'$ (see equation~\eqref{KKn}) is the limit of the semi-analytic set $K_n'$ (see Definition~\ref{defn of Kn}), independently of the choice of the integer $k_0 \ge n$.

Since the set $\tilde{K}_n$ is semi-analytic, its limit is a subset of $\mathcal{K}_n'$ consisting of points that satisfy inequalities~\eqref{eq:first}, \eqref{eq:second}, and \eqref{eq:third}, after replacing $t_n$ in equation~\eqref{tn} by its limiting value given in~\eqref{tn m3=0} and setting $r_3 = 1$. We show below that inequality~\eqref{eq:fourth} can be omitted when $r_3$ is sufficiently close to $1$.

Specifically, by using~\eqref{(r1, r2)}, inequalities~\eqref{eq:first}, \eqref{eq:second}, and \eqref{eq:third} become
\begin{align}
X &\ge \frac{1}{n}, \label{eqq:first} \\
2nX - n(n+1)Y - 2 &\ge 0, \label{eqq:second} \\
nX - n(n+1)Y - 1 &\ge 0, \label{eqq:third}
\end{align}
respectively. Evidently, the inequality $X \ge \frac{2}{n}$ implies~\eqref{eqq:first}, and inequality~\eqref{eqq:third} implies~\eqref{eqq:second}. Therefore, for given $n \ge 2$, the limit of the semi-analytic set $\tilde{K}_n$ is precisely the set $\mathcal{P}_n$.

We now discuss inequality~\eqref{eq:fourth} when $r_3$ is close to $1$. Define
\[
H'(r_1, r_2, r_3, n) = v_n - f(1).
\]
Since $v_{k_0+2} \ge v_{n+2}$, it follows that if $H'(r_1, r_2, r_3, n+2) \ge 0$, then inequality~\eqref{eq:fourth} holds. Under the coordinate transformation~\eqref{(r1, r2)}, we denote $H'(r_1, r_2, r_3, n+2)$ by $H(X, Y, r_3, n+2)$. A direct computation yields 
\begin{equation}\label{H expression}
    \begin{aligned}
      H(X, Y, r_3, n+2)
      =\frac{2 |E| \left(r_3 + s_3 \right)^{n+2}-2 F \left(r_3 - s_3 \right)^{n+2}}
      {\left(|E|+F \right)\left(\left(r_3 - s_3 \right)^{n+2} +\left(r_3 + s_3 \right)^{n+2}\right)},
    \end{aligned}
\end{equation}
where
\[
E=\left(r_3 - s_3 \right)\sqrt{\frac{X + Y + 1}{Y}}-\sqrt{\frac{Y + 1}{Y}},\quad
F=\left(r_3 + s_3 \right)\sqrt{\frac{X + Y + 1}{Y}}-\sqrt{\frac{Y + 1}{Y}}.
\]
It is evident that $H(X, Y, 1, n+2) = 0$. From the definition of the set $\tilde{K}_n$, we always have $\frac{2}{n} \le X \le \frac{2}{n-1}$ and
\begin{equation}\label{Y inequality}
0<Y \le \frac{nX-1}{n(n+1)}
\le \frac{\frac{2n}{n-1}-1}{n(n+1)}
= \frac{1}{n(n-1)},
\end{equation}
for $n \ge 2$.
It is not difficult to see that $E \ge 0$ if and only if
$$
(r_3 - s_3)^2X + \left((r_3 - s_3)^2 - 1\right)(Y + 1) \ge 0.
$$
Since $X$ admits positive lower and upper bounds and $Y$ admits an upper bound, this implies that $E \ge 0$ when $r_3$ is sufficiently close to $1$.
Clearly, the sign of $H(X, Y, r_3, n+2)$ is determined entirely by its numerator. After simplification, the numerator becomes
\[
\left(\left(r_3 + s_3 \right)^{n+1}-\left(r_3 - s_3 \right)^{n+1}\right)
\left(\sqrt{\frac{X + Y + 1}{Y}}-2 s_3 \sqrt{\frac{Y + 1}{Y}}\right),
\]
which is positive for $m_3$ sufficiently small.

For~$\mathrm{(b)}$, let $\mathcal{CP}_1$ be an arbitrary compact subset of $\mathcal{P}_1$. This is equivalent to saying that $\mathcal{CP}_1$ is a subset of $\mathcal{P}_1$ satisfying an additional compactness condition on $(X, Y)$. We then define a compact subset $CK_1$ of $K_1$ as the set of points satisfying the same compactness condition on $(X, Y)$.

The argument for conclusion~$\mathrm{(b)}$ is essentially the same as that for~$\mathrm{(a)}$. 
Here, the compactness condition on $(X, Y)$ guarantees the existence of positive lower and upper bounds for $X$, as well as an upper bound for $Y$. This, in turn, implies that $E \ge 0$ in equation~\eqref{H expression} when $r_3$ is sufficiently close to $1$. Consequently, the limit of the semi-analytic set $CK_1$ is the set $\mathcal{CP}_1$.   
\end{proof}

\medskip

\section*{Acknowledgments}
We would like to express our sincere gratitude to the referees for their valuable suggestions. As a result of these insightful comments, the manuscript has been significantly improved.

This work was supported by the Guangdong Basic and Applied Basic Research Foundation (No. 2025A1515011486) and the National Natural Science Foundation of China (No.12271148).


\begin{thebibliography}{999}

\bibitem[1]{bk} 
A.~Basmajian and Y.~Kim, 
\newblock \emph{Tubes in complex hyperbolic manifolds}, 
\newblock Trans. Amer. Math. Soc., 
\newblock {\bf 378}, 2031--2060 (2025).

\bibitem[2]{bea} 
A.~F.~Beardon, 
\newblock \emph{The geometry of discrete groups}, 
\newblock Graduate Texts in Mathematics 91, Springer-Verlag, New York, (1983).



\bibitem[3]{der} 
    M.~Deraux,
    \newblock \emph{Deforming the $\mathbb{R}$-Fuchsian $(4,4,4)$-triangle group into a lattice},
    \newblock Topology,
    \newblock {\bf 45}, 989--1020 (2006).


    \bibitem[4]{dpp1} 
	M.~Deraux, J.R.~Parker and J.~Paupert,
	\newblock \emph{New non-arithmetic complex hyperbolic lattices},
	\newblock Invent. Math.,
	\newblock {\bf 203}, 681--771 (2016).


     \bibitem[5]{dpp2} 
	M.~Deraux, J.R.~Parker and J.~Paupert,
	\newblock \emph{New Nonarithmetic Complex Hyperbolic Lattices $\mathrm{II}$},
	\newblock Michigan Math. J.,
	\newblock {\bf 70}(1), 133--205 (2021).

\bibitem[6]{egms} 
A.~Elzenaar, J.~Gong, G.~J.~Martin, and J.~Schillewaert, 
\newblock \emph{Bounding deformation spaces of Kleinian groups with two generators}, 
\newblock arXiv:2405.15970 (2024).



\bibitem[7]{gil} 
J.~Gilman, 
\newblock \emph{A discreteness condition for subgroups of $\mathrm{PSL}(2, \mathbb{C})$}, 
\newblock Contemp. Math., 
\newblock {\bf 211}, 261--267 (1997).

\bibitem[8]{gk} 
J.~Gilman and L.~Keen,
\newblock \emph{The geometry of two generator groups: hyperelliptic handlebodies}, 
\newblock Geom. Dedicata, 
\newblock {\bf 110}, 159--190 (2005).

\bibitem[9]{gp} 
W.~M.~Goldman and J.~R.~Parker, 
\newblock \emph{Complex hyperbolic ideal triangle groups}, 
\newblock J. Reine Angew. Math., 
\newblock {\bf 425}, 71--86 (1992).

\bibitem[10]{gol} 
W.~M.~Goldman, 
\newblock \emph{Complex hyperbolic geometry}, 
\newblock Oxford University Press, New York, (1999).

\bibitem[11]{gro} 
C.~H.~Grossi, 
\newblock \emph{On the type of triangle groups}, 
\newblock Geom. Dedicata, 
\newblock {\bf 130}, 137--148 (2007).

\bibitem[12]{jwx} 
Y.~Jiang, J.~Wang, and B.~Xie, 
\newblock \emph{A uniformizable spherical CR structure on a two-cusped hyperbolic 3-manifold}, 
\newblock Algebr. Geom. Topol., 
\newblock {\bf 23}(9), 4143--4184 (2023).

\bibitem[13]{jx} 
C.~Jiang and B.~Xie, 
\newblock \emph{Free groups generated by two unipotent maps}, 
\newblock Forum. Math., 
\newblock {\bf 37}(2), 443--458 (2025).


\bibitem[14]{kp} 
S.~B.~Kalane and J.~R.~Parker, 
\newblock \emph{Free groups generated by two parabolic maps}, 
\newblock Math. Z., 
\newblock {\bf 303}, 9 (2023).


\bibitem[15]{kpt1} 
    S.~Kamiya, J. R.~Parker, J. M.~Thompson,
    \newblock \emph{Notes on complex hyperbolic triangle groups},
    \newblock Conform. Geom. Dyn.,
    \newblock {\bf 14}, 202--218 (2010).

\bibitem[16]{kpt2} 
    S.~Kamiya, J. R.~Parker, J. M.~Thompson,
    \newblock \emph{Non-discrete complex hyperbolic triangle groups of type $(n, n, \infty; k)$},
    \newblock Canad. Math. Bull.,
    \newblock {\bf 55}(2), 329--338 (2012).




\bibitem[17]{lx} 
W.~Liao and B.~Xie, 
\newblock \emph{Free groups generated by two screw parabolic maps}, 
\newblock Geom. Dedicata, 
\newblock {\bf 218}, 107 (2024).

\bibitem[18]{mas} 
B.~Maskit, 
\newblock \emph{Kleinian groups}, 
\newblock Springer, New York, (1987).

\bibitem[19]{mon} 
A.~Monaghan, 
\newblock \emph{Complex hyperbolic triangle groups}, 
\newblock Ph.D. Thesis, University of Liverpool (2013).

\bibitem[20]{mpp} 
A.~Monaghan, J.~R.~Parker, and A.~Pratoussevitch, 
\newblock \emph{Discreteness of ultra-parallel complex hyperbolic triangle groups of type $[m_1, m_2, 0]$}, 
\newblock J. Lond. Math. Soc., 
\newblock {\bf 100}, 545--567 (2019).





\bibitem[21]{par} 
J.~R.~Parker, 
\newblock \emph{Notes on complex hyperbolic geometry}, 
\newblock University of Durham, \url{https://maths.dur.ac.uk/users/j.r.parker/img/NCHG.pdf} (2003).




\bibitem[22]{par2} 
    J. R.~Parker,
    \newblock \emph{Unfaithful complex hyperbolic triangle groups I: Involutions},
    \newblock Pac. J. Math.,
    \newblock {\bf 238}, 145--169 (2008).

    

\bibitem[23]{pp} 
    J.~R.~Parker and J.~Paupert,
    \newblock \emph{Unfaithful complex hyperbolic triangle groups II: Higher order reflections},
    \newblock Pac. J. Math.,
    \newblock {\bf 239}, 357--389 (2009).



\bibitem[24]{pwx} 
J.~R.~Parker, J.~Wang, and B.~Xie, 
\newblock \emph{Complex hyperbolic $(3, 3, n)$ triangle groups}, 
\newblock Pac. J. Math., 
\newblock {\bf 280}, 433--453 (2016).

\bibitem[25]{pw} 
J.~R.~Parker and P.~Will, 
\newblock \emph{A complex hyperbolic Riley slice}, 
\newblock Geom. Topol., 
\newblock {\bf 21}, 3391--3451 (2017).

\bibitem[26]{pov} 
S.~Povall, 
\newblock \emph{Ultra-parallel complex hyperbolic triangle groups}, 
\newblock Ph.D. Thesis, University of Liverpool (2019).

\bibitem[27]{pp1} 
S.~Povall and A.~Pratoussevitch, 
\newblock \emph{Complex hyperbolic triangle groups of type $[m, m, 0; 3, 3, 2]$}, 
\newblock Conform. Geom. Dyn., 
\newblock {\bf 24}, 51--67 (2020).

\bibitem[28]{pp2} 
S.~Povall and A.~Pratoussevitch, 
\newblock \emph{Complex hyperbolic triangle groups of type $[m, m, 0; n_1, n_2, 2]$}, 
\newblock Geom. Dedicata, 
\newblock {\bf 219}, 22 (2025).


\bibitem[29]{pra} 
A.~Pratoussevitch, 
\newblock \emph{Traces in complex hyperbolic triangle groups}, 
\newblock Geom. Dedicata, 
\newblock {\bf 111}, 159--185 (2005).

\bibitem[30]{pur} 
N.~Purzitsky, 
\newblock \emph{All Two-Generator Fuchsian Groups}, 
\newblock Math. Z., 
\newblock {\bf 147}, 87--92 (1976).

\bibitem[31]{rxj} 
X.~Ren, B.~Xie, and Y.~Jiang, 
\newblock \emph{A discreteness condition for subgroups of $\mathrm{PU}(2, 1)$}, 
\newblock Comput. Methods Funct. Theory, 
\newblock {\bf 19}, 411--431 (2019).

\bibitem[32]{ros} 
G.~Rosenberger, 
\newblock \emph{All generating pairs of all two-generator Fuchsian groups}, 
\newblock Arch. Math., 
\newblock {\bf 46}, 198--204 (1986).



\bibitem[33]{sch1} 
R.~E.~Schwartz, 
\newblock \emph{Ideal triangle groups, dented tori, and numerical analysis}, 
\newblock Ann. of Math., 
\newblock {\bf 153}, 533--598 (2001).

\bibitem[34]{sch2} 
R.~E.~Schwartz, 
\newblock \emph{Degenerating the complex hyperbolic ideal triangle groups}, 
\newblock Acta Math., 
\newblock {\bf 186}, 105--154 (2001).

\bibitem[35]{sch3} 
R.~E.~Schwartz, 
\newblock \emph{A better proof of the Goldman-Parker conjecture}, 
\newblock Geom. Topol., 
\newblock {\bf 9}, 1539--1601 (2005).

\bibitem[36]{sch4} 
R.~E.~Schwartz, 
\newblock \emph{Complex hyperbolic triangle groups}, 
\newblock In: Proceedings of the International Congress of Mathematicians, Vol. II (T.~Li, editor), Higher Ed. Press, Beijing, 339--349 (2002).

\bibitem[37]{sch5} 
R.~E.~Schwartz, 
\newblock \emph{Real hyperbolic on the outside, complex hyperbolic on the inside}, 
\newblock Invent. Math., 
\newblock {\bf 151}, 221--295 (2003).

\bibitem[38]{tho} 
    J. M.~Thompson,
    \newblock \emph{Complex hyperbolic triangle groups},
    \newblock Ph.D. Thesis, Durham University (2010).

\bibitem[39]{vas} 
S.~G.~de Assis Vasconcelos, 
\newblock \emph{Discretude de Grupos Triangulares Ultraparallelos em Geometria Hiperbolica Complexa (available in Portuguese only)}, 
\newblock Ph.D. Thesis, Universidade Federal de Minas Gerais, Brazil, \url{http://www.mat.ufmg.br/intranet-atual/pgmat/TesesDissertacoes/uploaded/Tese010.pdf} (2007).


\bibitem[40]{wg} 
J.~Wyss-Gallifent, 
\newblock \emph{Complex hyperbolic triangle groups}, 
\newblock Ph.D. Thesis, University of Maryland (2000). 

\bibitem[41]{xj} 
B.~Xie and Y.~Jiang, 
\newblock \emph{Groups generated by two elliptic elements in $\mathrm{PU}(2, 1)$}, 
\newblock Linear Algebra Appl., 
\newblock {\bf 433}, 2168--2177 (2010).

\bibitem[42]{xwx} 
M.~Xu, J.~Wang, and B.~Xie, 
\newblock \emph{Complex hyperbolic $(3, n,\infty)$ triangle groups}, 
\newblock J. Math. Anal. Appl., 
\newblock {\bf 492}(1), 124409 (2020).


\end{thebibliography}
\end{document}